\documentclass[letter,final,oneside,notitlepage,11pt]{article}

\usepackage{amssymb}
\usepackage{amsmath}
\usepackage{amstext}
\usepackage[]{geometry}
\usepackage{mathrsfs}
\usepackage{euscript}
\usepackage[all]{xy}
\usepackage{xstring}
\usepackage{slashed}

\def\N {\mathbb{N}}
\def\Z {\mathbb{Z}}

\def\R {\mathbb{R}}
\def\C {\mathbb{C}}

\def\id{\mathrm{id}}

\def\trivlin{\mathbf{I}}

\def\quand{\quad\text{ and }\quad}
\def\quomma{\quad\text{, }\quad}

\def\ev{\mathrm{ev}}

\def\act#1#2{#1/\!\!/#2}

\def\hc#1{\mathrm{h}_{#1}}

\def\subset{\subseteq}

\def\nobr{~\hspace{-0.26em}}
\def\maps{\nobr:\nobr}

\def\eq{\nobr = \nobr}
\def\pr{{\mathsf{pr}}}

\let\Oldin\in\renewcommand{\in}{\nobr\Oldin\nobr}
\let\Oldtimes\times\renewcommand{\times}{\nobr\Oldtimes}
\let\Oldotimes\otimes\renewcommand{\otimes}{\nobr\Oldotimes}

\newlength{\widthtmp}
\def\length#1{\settowidth{\widthtmp}{#1}\the\widthtmp}

\def\lli#1{\,_{#1}\!}

\renewcommand{\varepsilon}{\epsilon}
\makeatletter
\def\bigset#1#2{\left\lbrace\;\begin{minipage}[c]{#1}\begin{center}#2\end{center}\end{minipage}\;\right\rbrace}

\makeatother

\def\erf#1{(\ref{#1})}

\newlength{\myl}

\newcommand{\ueins}{{\mathrm{U}}(1)}

\newcommand{\spin}[1]{{\mathrm{Spin}}\brackets{#1}}
\newcommand{\spinc}[1]{{{\mathrm{Spin}}^{{\C}}}\brackets{#1}}
\newcommand{\so}[1]{{\mathrm{SO}}\brackets{#1}}

\def\diff{\mathcal{D}\!i\!f\!\!f}

\def\hom{\mathcal{H}\!om}
\def\homcon{\mathcal{H}\!om^{\!\nabla}\!}

\def\brackets#1{\IfStrEq{#1}{-}{}{(#1)}}
\def\buntech#1#2{\mathcal{B}\hspace{-0.01em}un_{\hspace{-0.1em}#1}^{#2}}
\def\bun#1#2{\buntech{#1}{}\brackets{#2}}
\def\buncon#1#2{\buntech{#1}{\nabla}\hspace{-0.05em}\brackets{#2}}

\def\grbtech#1{\mathcal{G}\hspace{-0.06em}r\hspace{-0.06em}b_{\hspace{-0.07em}#1}}

\def\grbcon#1#2{\grbtech{#1}^{\nabla\!}\brackets{#2}}

\newcommand{\alxydim}[2]{\begin{aligned}\xymatrix#1{#2}\end{aligned}}

\renewcommand{\to}{\nobr\!\xymatrix@R=0cm@C=1.4em{\ar[r] &}\nobr}
\renewcommand{\mapsto}{\!\xymatrix@R=0cm@C=1.4em{\ar@{|->}[r] &}\!}
\renewcommand{\Rightarrow}{\!\xymatrix@R=0cm@C=1.4em{\ar@{=>}[r] &}\!}
\renewcommand{\Leftarrow}{\!\xymatrix@R=0cm@C=1.4em{\ar@{<=}[r] &}\!}
\newcommand{\incl}{\!\xymatrix@R=0cm@C=1.4em{\ar@{^(->}[r] &}\!}

\renewcommand\Leftrightarrow{\!\xymatrix@R=0cm@C=1.4em{\ar@{<=>}[r] &}\!}

\usepackage{amsthm}
\usepackage{graphicx}
\usepackage[margin=2cm,font=normalsize,labelfont=bf]{caption}
\usepackage{verbatim}
\usepackage{enumerate}
\usepackage{fancyhdr}
\usepackage[]{geometry}
\usepackage[normalem]{ulem}
\usepackage{xstring}
\usepackage{needspace}

\makeatletter

\newcounter{denseversion}
\newcounter{authorcounter}
\newcounter{adresscounter}

\def\title#1{\gdef\@title{#1}}
\def\@title{}
\def\subtitle#1{\gdef\@subtitle{#1}}
\def\@subtitle{}

\def\authortagsused{0}
\def\adresstag#1{\if!#1!\else$^{\;#1\;}$\fi}

\renewcommand{\author}[2][]{
  \stepcounter{authorcounter}
  \if!#1!\else\gdef\authortagsused{1}\fi
  \ifnum\value{authorcounter}=1
    \def\@authorstringa{#2\adresstag{#1}}
    \def\@authorstringb{#2}
    \def\@authorstringc{#2\adresstag{#1}}
  \else
    \g@addto@macro\@authorstringa{\ and #2\adresstag{#1}}
    \g@addto@macro\@authorstringb{\ and #2}
    \g@addto@macro\@authorstringc{\\#2\adresstag{#1}}
  \fi}
\def\@author{\ifnum\value{denseversion}=0\@authorstringa\else\@authorstringb\fi}

\def\@adressstringa{}
\def\@adressstringb{}
\newcommand{\adress}[2][]{
  \stepcounter{adresscounter}
  \ifnum\value{adresscounter}=1
    \g@addto@macro\@adressstringa{\ifnum\authortagsused=0\def\br{\\}\else\def\br{, }\fi\adresstag{#1}#2}
    \g@addto@macro\@adressstringb{\def\br{\\}\adresstag{#1}\parbox[t]{14cm}{#2}}
  \else
    \g@addto@macro\@adressstringa{\\[\bigskipamount]\adresstag{#1}#2}
    \g@addto@macro\@adressstringb{\\[\medskipamount]\adresstag{#1}\parbox[t]{14cm}{#2}}
  \fi}
\def\@adress{\ifnum\value{denseversion}=0\@adressstringa\else\@adressstringb\fi}

\def\preprint#1{\gdef\@preprint{#1}}
\def\@preprint{}
\def\keywords#1{\gdef\@keywords{#1}}
\def\@keywords{}
\def\msc#1{\gdef\@msc{#1}}
\def\@msc{}
\def\email#1{
   \gdef\@email{#1}
   \g@addto@macro\@authorstringc{ {\it (#1)}}}
\def\@email{}
\def\dedication#1{\gdef\@dedication{#1}}
\def\@dedication{}

\def\mybaselinestretch#1{\gdef\@mybaselinestretch{#1}}
\def\@mybaselinestretch{}

\def\refname{References}

\mybaselinestretch{1.2}
\renewcommand{\baselinestretch}{\@mybaselinestretch}

\def\denseversion{
  \setcounter{denseversion}{1}
  \newgeometry{left=2cm,right=2cm,top=2cm}
  \mybaselinestretch{1.1}
  \renewcommand{\baselinestretch}{\@mybaselinestretch}
  \normalfont
  \fancyfoot[C]{\itshape{\hspace{2.5cm}--$\,\,$\thepage$\,\,$--}}}

\newlength{\myparskip}
\newlength{\myproofparskip}

\renewcommand{\emph}[1]{\def\reserved@a{it}\ifx\f@shape\reserved@a\uline{#1}\else\textit{#1}\fi}

\newcommand{\mytableofcontents}{
   \ifnum\value{denseversion}=0
     \tableofcontents
   \else
     \renewcommand{\baselinestretch}{0.8}
     \normalfont
     \tableofcontents
     \renewcommand{\baselinestretch}{\@mybaselinestretch}
     \normalfont
   \fi}

\newlength{\zeilenlaenge}
\def\putindent#1{
  \settowidth{\zeilenlaenge}{#1}
  \ifnum\zeilenlaenge>\textwidth
    #1
  \else
    \noindent #1
  \fi
}

\def\href#1#2{#2}

\def\kohyp{
  \usepackage{hyperref}
  \hypersetup{
    linktocpage = true,
    pdftitle = {\@title},
    pdfauthor = {\@author},
    pdfkeywords = {\@keywords},    
    bookmarksopen = true,
    bookmarksopenlevel = 1
  }}  
\def\showkeywords{\begin{flushleft}\footnotesize\textbf{Keywords}: \@keywords\end{flushleft}}
\def\showmsc{\begin{flushleft}\footnotesize\textbf{MSC 2010}: \@msc\end{flushleft}}

\newcounter{mythm}[subsection]
\newcounter{mainthm}

\def\setsecnumdepth#1{
  \setcounter{secnumdepth}{#1}
  \setcounter{mythm}{0}
  \ifnum \c@secnumdepth >0
    \ifnum \c@secnumdepth >1
      \def\themythm{\thesubsection.\arabic{mythm}}
      \numberwithin{equation}{subsection}
      \renewcommand\theequation{\thesubsection.\arabic{equation}}
    \else
      \def\themythm{\thesection.\arabic{mythm}}
      \numberwithin{equation}{section}
      \renewcommand\theequation{\thesection.\arabic{equation}}
    \fi
  \else
    \def\themythm{\arabic{mythm}}
  \fi}

\newenvironment{mythmenv}{\strut\ \setlength{\parskip}{\myproofparskip}}{\setlength{\parskip}{\myparskip}}

\newlength{\mythmskip}
\newlength{\mythmtopskip}
\newtheoremstyle{mythmstylea}{\mythmtopskip}{\mythmskip}{\it}{}{\bf}{.}{0em}{}
\newtheoremstyle{mythmstyleb}{\mythmtopskip}{\mythmskip}{}{}{\bf}{.}{0em}{}

\theoremstyle{mythmstylea}
\newtheorem{mytheorem}[mythm]{Theorem}
\newtheorem{mydefinition}[mythm]{Definition}
\newtheorem{mycorollary}[mythm]{Corollary}
\newtheorem{myproposition}[mythm]{Proposition}
\newtheorem{mylemma}[mythm]{Lemma}

\newtheorem{mymaintheorem}[mainthm]{Theorem}
\newtheorem{mymaincorollary}[mainthm]{Corollary}
\newtheorem{mymainproposition}[mainthm]{Proposition}
\newtheorem{mymaindefinition}[mainthm]{Definition}

\theoremstyle{mythmstyleb}
\newtheorem{myremark}[mythm]{Remark}
\newtheorem{myexample}[mythm]{Example}
\newtheorem{myexercise}[mythm]{Exercise}

\newenvironment{definition}[1][]{\begin{mydefinition}[#1]\begin{mythmenv}}{\end{mythmenv}\end{mydefinition}}

\newenvironment{proposition}[1][]{\begin{myproposition}[#1]\begin{mythmenv}}{\end{mythmenv}\end{myproposition}}
\newenvironment{lemma}[1][]{\begin{mylemma}[#1]\begin{mythmenv}}{\end{mythmenv}\end{mylemma}}
\newenvironment{remark}[1][]{\begin{myremark}[#1]\begin{mythmenv}}{\end{mythmenv}\end{myremark}}

\newenvironment{maintheorem}[1]{\def\themainthm{#1}\begin{mymaintheorem}\begin{mythmenv}}{\end{mythmenv}\end{mymaintheorem}}
\newenvironment{maindefinition}[1]{\def\themainthm{#1}\begin{mymaindefinition}\begin{mythmenv}}{\end{mythmenv}\end{mymaindefinition}}
\newenvironment{maincorollary}[1]{\def\themainthm{#1}\begin{mymaincorollary}\begin{mythmenv}}{\end{mythmenv}\end{mymaincorollary}}

\renewenvironment{proof}[1][Proof]{\noindent #1. \begin{mythmenv}}{\hfill$\square$\end{mythmenv}\medskip}

\def\tocsection#1{\section*{#1}\addcontentsline{toc}{section}{#1}}

\def\mytitle{}
\def\zmptitle{
  \begin{tabular}{cc}
    \begin{minipage}[c]{0.4\textwidth}
      \begin{flushleft}
        \includegraphics[width=110pt]{../../tex/zmp}
      \end{flushleft}  
    \end{minipage}&
    \begin{minipage}[c]{0.55\textwidth}
      \begin{flushright}
      {\small\sf\@preprint}
      \end{flushright}
    \end{minipage}
  \end{tabular}
  \vskip 2cm}

\def\maketitle{
  \setlength{\parskip}{\myparskip}  
  \newpage
  \noindent
  \mytitle
  \begin{center}
    \LARGE\@title\\
    \if!\@subtitle!\else \smallskip\LARGE\@subtitle\\\fi
    \bigskip
    \if!\@author!\else\bigskip\large\@author\\\fi
    \ifnum\value{denseversion}=0
      \if!\@adress!\else     \bigskip\normalsize\@adress\\\fi
      \if!\@email!\else\ifnum\value{authorcounter}=1\bigskip\normalsize\textit{\@email}\\\else\fi\fi
    \else
    \fi
    \if!\@dedication!\else \bigskip\normalsize{\@dedication}\\\fi
  \end{center}
  \ifnum\value{denseversion}=0\vskip 1.5cm\else\vskip0.5cm\fi
  \thispagestyle{empty}}

\def\kobiburl#1{
   \IfBeginWith
     {#1}
     {http://arxiv.org/abs/}
     {\kobibarxiv{#1}}
     {\kobiblink{#1}}}
\def\kobibarxiv#1{\href{#1}{\texttt{[arxiv:\StrGobbleLeft{#1}{21}]}}}
\def\kobiblink#1{Available as: \href{#1}{\texttt{\StrSubstitute{#1}{_}{\underline{\;\;}}}}}

\newcommand{\etalchar}[1]{$^{#1}$}
\def\kobib#1{
  \begin{raggedright}
  \ifnum\value{denseversion}=0\else\small\fi

  \end{raggedright}
  \ifnum\value{denseversion}=0\else
      \noindent
      \if!\@authorstringc!\else
        \ifnum\authortagsused=0\ifnum\value{authorcounter}>1\normalsize\@authorstringc\\[\medskipamount]\else\fi\else\normalsize\@authorstringc\\[\medskipamount]\fi       \fi
      \if!\@adress!\else\normalsize\@adress\\\fi
      \ifnum\authortagsused=0\ifnum\value{authorcounter}=1\if!\@email!\else\linebreak\normalsize\textit{\@email}\\\fi\else\fi\else\fi
  \fi}

\newenvironment{commentfigure}{}
\newenvironment{sidewayscommentfigure}{\begin{minipage}}{\end{minipage}}

\def\showcomments{ -- Comments suppressed}

\newif\if@fewtab\@fewtabtrue{
  \count255=\time\divide\count255 by 60
  \xdef\hourmin{\number\count255}
  \multiply\count255 by-60\advance\count255 by\time
  \xdef\hourmin{\hourmin:\ifnum\count255<10 0\fi\the\count255}}
\def\ps@draft{
  \let\@mkboth\@gobbletwo
  \def\@oddfoot{
    \hbox to 7 cm{\tiny \versionno\hfil}
    \hskip -7cm\hfil\rm\thepage\hfil{\tiny\draftdate}}
  \def\@oddhead{}
  \def\@evenhead{}
  \let\@evenfoot\@oddfoot}
\def\draftdate{\number\month/\number\day/\number\year\ \ \ \hourmin }
\newcommand\version[1]{
  \typeout{}\typeout{#1}\typeout{}
  \vskip-1.7cm \centerline{\fbox{{\normalsize\tt DRAFT -- #1 -- 
  \draftdate\showcomments}}} \vskip0.92cm}
\def\draft#1{
  \def\versionno{#1}
  \pagestyle{draft}\thispagestyle{draft}
  \gdef\@ntitle{\version\versionno \@title}
  \global\def\draftcontrol{1}}
\global\def\draftcontrol{0}

\makeatother

\usepackage[latin1]{inputenc}
\usepackage[english]{babel}

\def\quot#1{``#1''}

\def\exd#1{{#1^{\vee}}}

\def\sufi{th}
\def\suficonn{\sufi^{\nabla}}

\def\thinpairs#1{#1^2_{\text{\tiny \it thin}}}
\def\thingrpd#1#2{\mathfrak#1#2}

\def\px#1#2{P_{\!#2}#1}
\def\p{P}

\def\ptx#1#2{\mathcal{P}_{\!#2}#1}
\def\ev{\mathrm{ev}}
\def\pt#1{\mathcal{P}#1}

\def\hc#1{\mathrm{h}_{#1}}
\def\pcomp{\nobr\star\nobr}

\def\prev#1{\overline{#1}}

\def\un{\mathscr{R}}
\def\uncon{\mathscr{R}^{\nabla}}
\def\tr{\mathscr{T}}
\def\trcon{\mathscr{T}^{\nabla}}

\def\fus#1#2{\mathcal{F}\!us(#2,#1)}
\def\fuslc#1#2{\mathcal{F}\!us_{\mathrm{lc}}(#2,#1)}
\def\fushom#1#2{h\!\fus {#1} {#2}}

\def\fusbun#1#2{\mathcal{F}\!us\buntech#1{}(#2)}
\def\fusbunconsf#1#2{\mathcal{F}\!us\bunconsf{#1}{#2}}
\def\fusbuncon#1#2{\mathcal{F}\!us\buncon{#1}{#2}}

\def\fusbunth#1#2{\mathcal{F}\!us\bunth{#1}{#2}}

\def\Omegafus{\Omega_{\!f\!u\!s}}

\def\bunconsf#1#2{\buntech{#1}{\,\nabla}{}^{_{\!\!s\!f}}\hspace{-0.15em}\brackets{#2}}

\def\bunth#1#2{\buntech{#1}{\!\!\!\;\;t\hspace{-0.03em}h}\hspace{-0.05em}\brackets{#2}}

\def\diffbun#1#2{\diff\hspace{-0.1em}\bun{#1}{#2}}

\def\diffgrb#1#2{\diff\hspace{-0.1em}\grbtech#1(#2)}
\def\diffgrbcon#1#2{\diff\hspace{-0.1em}\grbcon{#1}{#2}}

\def\ptr#1{\tau_{#1}}

\def\struc#1#2{#1\text{-}\mathcal{L}\!i\!f\!t(#2)}
\def\struccon#1#2#3{#1\text{-}\mathcal{L}\!i\!f\!t^{\nabla\!}_{#2}(#3)}

\title{Transgression to Loop Spaces and its Inverse, III}
\subtitle{Gerbes and Thin Fusion Bundles}
\author{Konrad Waldorf}
\adress{Fakultät für Mathematik\\Universität Regensburg\\Universitätsstraße 31\\D--93053 Regensburg}
\email{konrad.waldorf@mathematik.uni-regensburg.de}
\keywords{gerbes, diffeological spaces, holonomy, loop space, orientability of loop spaces, complex spin structures, transgression}
\msc{Primary 53C08,  Secondary 53C27, 55P35}

\kohyp

\usepackage{amstext}
\usepackage{graphicx}
\usepackage{amssymb}
\usepackage{amsmath}
\usepackage[normalem]{ulem}

\let\Oldbibitem\bibitem
\def\bibitem[#1]#2{
  \def\standardlabel{#1}
  \IfStrEq{\standardlabel}{Wala}{\def\standardlabel{Part I}}{}
  \IfStrEq{\standardlabel}{Walb}{\def\standardlabel{Part II}}{}
  \Oldbibitem[\standardlabel]{#2}}

\begin{document}


\maketitle 

\begin{abstract}
We show that the category of abelian gerbes over a smooth manifold is equivalent to a  certain category of principal bundles over the free loop space. These principal bundles are equipped with fusion products and are  equivariant with respect to thin homotopies between loops. The equivalence is established by a functor called regression, and complements a similar equivalence for bundles and gerbes equipped with connections, derived previously in Part II of this series of papers. The two equivalences provide a complete  loop space formulation of the geometry of gerbes;   functorial, monoidal, natural in the base manifold, and consistent with passing from the setting \quot{with connections} to the one \quot{without connections}. We discuss an application to lifting problems, which provides in particular  loop space formulations of spin structures, complex spin structures, and spin connections.
\showkeywords
\showmsc
\end{abstract}

\newpage

\mytableofcontents

\def\noco#1{}

\section{Summary}

\label{sec:introduction}

This is the third and last part of a series of papers \cite{waldorf9,waldorf10} providing a complete formulation of the geometry of  abelian bundles and gerbes over a smooth manifold $M$ in terms of its free loop space $LM$. In this section we  give an overview about the contents and the results of these three papers. In Section \ref{sec:results} we give a more focused summary of the present Part III. This paper is written in a self-contained way -- contents taken from Parts I or II are either reviewed or explicitly referenced. For the convenience of the reader, we have included a table with the  notations of all three papers at the end of this paper.

\cite{waldorf9} is concerned with bundles over $M$. The bundles we consider are principal $A$-bundles over $M$, for $A$ an abelian Lie group, possibly discrete, possibly non-compact. 
The main result of Part I is a loop space formulation for the geometry of these bundles. It can be summarized in the  diagram 
\begin{equation}
\label{diag1}
\alxydim{@C=2cm@R=1.5cm}{\hc 0 \buncon A M  \ar[d]  \ar@<0.4em>[r]^-{\trcon} &  \fus A {\mathcal{L}M} \ar@<0.4em>[l]^{\uncon} \ar[d] \\ \hc 0 \bun AM \ar@<0.4em>[r]^{\tr} & h\fus A{\mathcal{L}M}\text{.} \ar@<0.4em>[l]^{\un}}
\end{equation}
In the left column, we have the categories $\buncon A M$ and $\bun AM$ of principal $A$-bundles over $M$ with and without connections, respectively,  the symbol $\hc 0$ denotes the operation of taking sets of isomorphism classes of objects, and the vertical arrow is the operation of forgetting  connections. The right column contains the corresponding loop space formulations: we have a set $\fus A {\mathcal{L}M}$ consisting of \emph{fusion maps} on the \emph{thin loop space} $\mathcal{L}M$ of $M$. Basically, this is a smooth map $f: LM \to A$ that is constant on thin homotopy classes of loops, and satisfies
\begin{equation}
\label{cond1}
f(\prev{\gamma_3} \pcomp \gamma_1) = f(\prev{\gamma_3} \pcomp \gamma_2) \cdot f(\prev{\gamma_2} \pcomp \gamma_1)
\end{equation}
whenever $\gamma_1,\gamma_2,\gamma_3$ is a triple of paths in $M$ with a common initial point and a common end point \cite[Definition 2.2.3]{waldorf9}. The set $h\fus A {\mathcal{L}M}$ consists of \emph{homotopy classes} of fusion maps, and the vertical arrow denotes the projection of a fusion map to its homotopy class.

The maps $\trcon$ and $\tr$ in diagram \erf{diag1} are called \emph{transgression}; they basically take the holonomy of a connection. The maps $\uncon$ and $\un$ in the opposite direction are called \emph{regression}; they  construct a principal bundle from a fusion map. The statement of Theorems A and B of \cite{waldorf9} is that the pairs $(\trcon,\uncon)$ and $(\tr,\un)$ form bijections. Thus, the sets $\fus A {\mathcal{L}M}$ and $h\fus A {\mathcal{L}M}$ are our loop space formulations of principal $A$-bundles with and without connection, respectively.
Theorem C of \cite{waldorf9} states  the commutativity of the diagram; it expresses the fact that these loop space formulations are compatible with going from a setup \emph{with} connections to one \emph{without} connections. Moreover, all arrows in diagram \erf{diag1} respect the group structures: in the left column the one induced by the monoidal structures on the categories  $\buncon A M$ and $\bun AM$, and in the right column the one given by point-wise multiplication of fusion maps. Finally, all sets and arrows in diagram \erf{diag1} are natural with respect to the manifold $M$.

The content of \cite{waldorf10} and the present Part III is a generalization of  diagram \erf{diag1} from bundles to gerbes. The gerbes we consider are \emph{diffeological $A$-bundle gerbes}, whose structure group $A$ is -- as before -- any abelian Lie group. The main results of Parts II and III can be summarized in the diagram
\begin{equation}
\label{diag2}
\alxydim{@C=2cm@R=1.5cm}{\hc 1\diffgrbcon A M \ar[d]\ar@<0.4em>[r]^{\trcon} &    \fusbunconsf A{LM} \ar@<0.4em>[l]^{\uncon_x}  \ar[d]^{\sufi} \\ \hc 1\diffgrb A M & \ar[l]^-{h\un_x}  h\fusbunth A {LM} }
\end{equation}
of categories and functors.
In the left column we have the 2-categories $\diffgrb AM$ and $\diffgrbcon AM$  of diffeological $A$-bundle gerbes without and with connections, respectively, the symbol $\hc 1$ denotes the operation of producing a category from a 2-category by identifying 2-isomorphic 1-morphisms, and the vertical arrow is the operation of forgetting  connections. The right column contains categories of principal $A$-bundles over $LM$, and the main
problem addressed in Parts II and III is to find  versions of these two categories such that the horizontal arrows are equivalences of categories.

Both categories of bundles over $LM$ are based  on the notion of a \emph{fusion product}, which generalizes condition \erf{cond1} from maps to bundles. Fusion products have been introduced in a slightly different form by Brylinski and McLaughlin \cite{brylinski4,brylinski1}, and have also been used in yet another  form by Stolz and Teichner \cite{stolz3}. Here, a fusion product on a principal $A$-bundle $P$ over $LM$ provides fibre-wise isomorphisms
\begin{equation*}
\lambda: P_{\prev{\gamma_2} \pcomp \gamma_1} \otimes P_{\prev{\gamma_3} \pcomp \gamma_2} \to P_{\prev{\gamma_3} \pcomp \gamma_1}
\end{equation*} 
for triples $\gamma_1,\gamma_2,\gamma_3$  of paths in $M$ with a common initial point and a common end point \cite[Definition 2.1.3]{waldorf10}. Principal $A$-bundles over $LM$ with fusion products are called \emph{fusion bundles} and form a category which we denote by $\fusbun A {LM}$. 
Fusion products are important because they furnish a functor
\begin{equation*}
\un_x: \fusbun A{LM} \to \hc 1\diffgrb AM\text{,}
\end{equation*}
which we call \emph{regression functor} \cite[Section 5.1]{waldorf10}. It depends -- up to natural equivalences -- on a base point $x\in M$. The two functors $\uncon_x$ and $h\un_x$ in diagram \erf{diag2} are variations  of this functor.

In the setup with connections -- corresponding to the first row in diagram \erf{diag2} and addressed in \cite{waldorf10} -- fusion bundles have to be equipped with \emph{superficial connections} \cite[Definition A]{waldorf10}, forming the category $\fusbunconsf A{LM}$.  We have introduced superficial connections as connections whose holonomy is subject to certain constraints. They permit to promote the regression functor to a setup \quot{with connections}, i.e. to a functor
\begin{equation*}
\uncon_x: \fusbunconsf A{LM} \to \hc 1 \diffgrbcon AM\text{.}
\end{equation*}
In the opposite direction, there exists a transgression functor $\trcon$ defined by Brylinski and McLaughlin \cite{brylinski1}. 
Theorem A of Part II states that the functors $\uncon_x$ and $\trcon$ form  an equivalence of categories. Thus, fusion bundles with superficial connections over $LM$ are our loop space formulation for gerbes with connections over $M$.

In the setup without connections -- corresponding to the second row in diagram \erf{diag2} and addressed in the present Part III -- fusion bundles undergo two modifications. A detailed account is given in the next section. First, we
equip them with a \emph{thin structure}, a kind of equivariant structure with respect to thin homotopies of loops. Fusion bundles with thin structure are called \emph{thin fusion bundles}.
Second, we
pass to the \emph{homotopy category}, i.e. we identify homotopic bundle morphisms.
The result is the category $h\fusbunth A {LM}$ in the bottom right corner of diagram \erf{diag2}.
 The vertical functor $\sufi$ in the second column of diagram \erf{diag2} produces (via parallel transport) a thin structure from a superficial connection, and projects a bundle morphism to its homotopy class. The regression functor $\un_x$ factors on the level of morphisms through homotopy classes, and so induces a functor \begin{equation}
\label{eq:hreg}
h\un_x: h\fusbunth A{LM} \to \hc1 \diffgrb AM\text{,}
\end{equation}
which we find in the bottom row of diagram \erf{diag2}. In contrast to the setup \quot{with connections}, there is no distinguished inverse functor. Still, Theorem \ref{th:mainA}  of the present article states that the functor \erf{eq:hreg} is an equivalence of categories. Thus, thin fusion bundles over $LM$ are our loop space formulation for gerbes over $M$.

Theorem \ref{th:mainB} of the present article states that diagram \erf{diag2} is (strictly) commutative -- it expresses the fact that our two loop space formulations, namely fusion bundles with superficial connections and thin fusion bundles, are compatible with passing from the setup \emph{with} connections to the one \emph{without} connections.  Finally, all categories and functors in  diagram \erf{diag2}  are monoidal and natural with respect to base point-preserving smooth maps.

Now that we have summarized the results of this project, let me indicate how they can be used. The typical application arises whenever some theory over a smooth manifold $M$ can be formulated in terms of categories of abelian gerbes or gerbes with connections. Then, our results allow to transgress the theory to an equivalent theory on the loop space $LM$, formulated in terms of fusion bundles with thin structures or superficial connections.

One example is \emph{lifting problems} -- problems of lifting the structure group of a principal bundle to an (abelian) central extension -- which can be formulated in terms of \emph{lifting bundle gerbes}. The case of spin structures was the initial motivation for this project: spin structures are lifts of the structure group of the frame bundle of an oriented Riemannian manifold $M$ from $\so -$ to $\spin -$. Stolz and Teichner have shown using Clifford bimodules that spin structures on $M$ are the same as \emph{fusion-preserving orientations} of $LM$. In \cite[Corollary E]{waldorf10} we have reproduced this result using transgression/regression for the \emph{discrete group} $A=\Z/2\Z$. In \cite[Corollary 6.3]{waldorf13} we have treated geometric lifting problems for general abelian Lie groups $A$, in particular leading to a  loop space formulation of spin connections (for $\spinc-$-structures). In Section \ref{sec:applications} of the present paper  we complete the discussion of lifting problems; see Theorems \ref{th:lifting} and \ref{th:liftingcompl}.

Another example are \emph{multiplicative gerbes} over a Lie group $G$ \cite{carey4}. Multiplicative gerbes transgress to central extensions of the loop group $LG$. As a result of this project, we can identify additional structure on these central extensions: fusion products, thin structures, and superficial connections. Some consequences of the existence of these structures have been elaborated in Section 1.3 of \cite{waldorf10}; other features are yet to be developed, for instance, the impact of fusion products  to the representation theory of central extensions of loop groups. Further possible things one could look at in this context are \emph{index gerbes} \cite{Lott2002,bunke2002} and twisted K-theory in its formulation by bundle gerbe modules \cite{bouwknegt1}.

An extension of the results of this paper series to gerbes and bundles with \emph{non-abelian} structure groups seems at present time far away. In joint work with Nikolaus \cite{Nikolaus,Nikolausa} we have managed to define a transgression map on the level of isomorphism classes of non-abelian gerbes, but we have not looked at fusion products, thin structures, or superficial connections on bundles with non-abelian structure groups.

\paragraph{Acknowledgements.}
I gratefully acknowledge  a Feodor-Lynen scholarship, granted by the Alex\-an\-der von Hum\-boldt Foundation.   I thank Thomas Nikolaus  for helpful discussions.

\section{Results of this Article}

\label{sec:results}

\subsection{Main Theorems}

\label{sec:maintheorem}

In the present article we are concerned with principal bundles over the free loop space $LM = C^{\infty}(S^1,M)$ of a smooth manifold $M$. We understand $LM$ as a \emph{diffeological space}, and use the theory of principal bundles over diffeological spaces developed in \cite[Section 3]{waldorf9}. 
The structure group of the bundles  is an abelian Lie group $A$, possibly discrete, possibly non-compact. The goal of this article is to specify additional structure for principal $A$-bundles over $LM$ such that the resulting category becomes equivalent to the category of diffeological $A$-bundle gerbes over $M$. A detailed discussion of this additional structure is the content of Section \ref{sec:fusbun}.

The first additional structure is a \emph{thin structure}, a central invention of this article and the content of Section \ref{sec:thin}.  We first introduce the notion of an almost-thin structure, and then formulate an integrability condition. Roughly, an \emph{almost-thin structure} provides fibre-wise identifications 
\begin{equation*}
d_{\tau_1,\tau_2}: P_{\tau_1} \to P_{\tau_2}
\end{equation*}  
where $\tau_1$ and $\tau_2$ are thin homotopic loops, i.e. there is a rank-one homotopy between them. The identifications $d_{\tau_1,\tau_2}$ are supposed to satisfy the cocycle condition
\begin{equation*}
d_{\tau_{2},\tau_{3}} \circ d_{\tau_1,\tau_2} = d_{\tau_1,\tau_3}
\end{equation*}
whenever $\tau_1$, $\tau_2$ and $\tau_3$ are thin homotopic. The crucial point is to specify in which way the identifications $d_{\tau_1,\tau_2}$ fit together into a \emph{smooth family}, i.e. to define a diffeology on the set $\thinpairs {LM}$ of pairs of thin homotopic loops.
Somewhat surprisingly, it turns out that the relevant diffeology is \emph{not} the subspace diffeology of $\thinpairs {LM} \subset LM \times LM$, but a finer one which allows to choose thin homotopies locally in smooth families. In more appropriate language, we introduce a diffeological groupoid $\mathfrak{L}M$ which we call the \emph{thin loop stack}; with objects $LM$ and morphisms $\thinpairs{LM}$. A principal $A$-bundle \emph{over the groupoid $\mathfrak{L}M$} is a pair $(P,d)$ of a principal $A$-bundle over $LM$ and an equivariant structure $d$ -- this equivariant structure is precisely what we call an almost-thin structure (Definition \ref{def:catalmost}).

A \emph{thin structure} is an almost-thin structure that satisfies an integrability condition which we formulate now. On bundles over the loop space $LM$ one can look at particular classes of connections. Here, the following class is relevant: a connection $\omega$ is called \emph{thin}, if its holonomy around a loop $\tau$ in $LM$ vanishes whenever the associated torus $S^1 \times S^1\to M$ is  of rank (at most) one. Equivalently, its parallel transport between two thin homotopic loops $\tau_1$ and $\tau_2$ is independent of the choice of a (rank one) path between them, and so determines  a well-defined  map $d^{\omega}_{\tau_1,\tau_2}:P_{\tau_1} \to P_{\tau_2}$. We prove that these maps form an almost-thin structure. The integrability condition for an almost-thin structure $d$ is that there exists a thin connection $\omega$ such that $d=d^{\omega}$ (Definition \ref{def:integability}).

The second additional structure is a \emph{fusion product}, which we have introduced in \cite[Definition 2.1.3]{waldorf10}.   We say that a path in $M$ is a smooth map $\gamma: [0,1] \to M$ with \quot{sitting instants}. These ensure that one can  compose two paths smoothly whenever the first ends where the second starts. The diffeological space of paths in $M$ is denoted by $PM$, and the end-point evaluation by $\ev:PM \to M \times M$. In diffeological spaces one can form the $k$-fold fibre products $PM^{[k]}$ of $PM$ with itself over the evaluation map. Then, we have a well-defined  smooth map
\begin{equation*}
l\maps PM^{[2]} \to LM\maps (\gamma_1,\gamma_2) \mapsto \prev{\gamma_2} \pcomp \gamma_1\text{.}
\end{equation*}
We denote by $e_{ij}$ the composition of $l$ with the projection $\mathrm{pr}_{ij}: PM^{[3]} \to PM^{[2]}$. Now, a \emph{fusion product} $\lambda$ on a principal $A$-bundle $P$ is a bundle morphism
\begin{equation*}
\lambda: e_{12}^{*}P \otimes e_{23}^{*}P \to e_{13}^{*}P
\end{equation*}
over $PM^{[3]}$ which satisfies an associativity constraint over $PM^{[4]}$.

Now we have described two additional structures for principal $A$-bundles over $LM$. We require that both are compatible in a certain way. In Part II we have already formulated compatibility conditions between a \emph{fusion product} and a \emph{connection}. Using these, we say that a \emph{compatible, symmetrizing thin structure} is an almost-thin structure which can be integrated to a compatible, symmetrizing thin connection (Definition \ref{def:compandsym}). More details about fusion products and  compatibility conditions are given in Section \ref{sec:fus}.

\begin{maindefinition}{A}
\label{def:fusbunth}
A \emph{thin fusion bundle} over $LM$ is a principal $A$-bundle equipped with a fusion product and a compatible, symmetrizing thin structure. \end{maindefinition}

Thin fusion bundles form a category $\fusbunth A {LM}$ whose morphisms are bundle morphisms that preserve the fusion products and the thin structures. It turns out that the category $\fusbunth A {LM}$ is not yet equivalent to the category $\hc 1 \diffgrb AM$ of diffeological $A$-bundle gerbes over $M$, as intended. This can be seen by looking at the Hom-sets. It is well-known that the Hom-category between two bundle gerbes forms a module over the monoidal category $\bun AM$ of principal $A$-bundles over $M$, in such a way that the associated Hom-set in $\hc 1 \diffgrb AM$ forms a torsor for the group $\hc 0 \bun AM$. On the other hand, the Hom-sets of $\fusbunth A{LM}$ form a torsor for the group $\fus A {\mathcal{L}M}$ of fusion maps. However,  the group $\hc 0 \bun AM$ is \emph{not} isomorphic to $\fus A{\mathcal{L}M}$, so that the two torsors cannot be in bijection. Instead, by \cite[Theorem B]{waldorf9} the group $\hc 0 \bun AM$ is isomorphic to the group $h\fus A {\mathcal{L}M}$ of \emph{homotopy classes} of fusion maps. This forces us to pass to the \emph{homotopy category} of thin fusion bundles.

\begin{maindefinition}{B}
\label{def:hfusbunth}
The \emph{homotopy category of thin fusion bundles} --  which we denote by $h\fusbunth A{LX}$ -- consists of thin fusion bundles and homotopy classes of bundle morphisms that preserve the fusion product and the thin structures. \end{maindefinition}

A detailed discussion of homotopy categories of bundles is the content of Section \ref{sec:homcat}. In order to establish the equivalence between the categories $\hc 1 \diffgrb AM$ and $h\fusbunth A {LM}$ we use the regression functor
\begin{equation*}
\un_x: \fusbun A {LM} \to \hc 1 \diffgrb AM
\end{equation*}
defined in \cite{waldorf10}, for which we prove that it factors -- on the level of morphisms -- through homotopy classes (Proposition \ref{prop:regfact}).  We denote the resulting functor by $h\un_x$; a detailed construction is given in Section \ref{sec:regression}.
The main theorem of this article is the following.

\begin{maintheorem}{A}
\label{th:mainA}
Let $M$ be a connected smooth manifold. 
Regression induces an equivalence of monoidal categories:
\begin{equation*}
h\un_x : h\fusbunth A {LM} \to \hc 1 \diffgrb A M \text{.}
\end{equation*}
Moreover, this equivalence is natural with respect to base point-preserving smooth maps.
\end{maintheorem}

We want to relate the equivalence of Theorem \ref{th:mainA} to the equivalence $\uncon_x$ in the setting with connections \cite[Theorem A]{waldorf10}. There, a category $\fusbunconsf A {LM}$ of fusion bundles with superficial connections is relevant. A \emph{superficial connection} is a thin connection with an additional property related to certain homotopies between loops in $LM$. In particular, a superficial connection $\omega$ defines a thin structure $d^{\omega}$, and since  connections on fusion bundles are (by definition) compatible and symmetrizing, the thin structure $d^{\omega}$ is also compatible and symmetrizing. This defines a functor
\begin{equation*}
\sufi: \fusbunconsf A {LM} \to h\fusbunth A {LM}\text{.}
\end{equation*}
The following theorem states that this functor corresponds under regression to the operation of forgetting  bundle gerbe connections. 

\begin{maintheorem}{B}
\label{th:mainB}
For any connected smooth manifold $M$, the diagram
\begin{equation*}
\alxydim{@C=2cm@R=1.5cm}{\fusbunconsf A {LM} \ar[r]^{\uncon_x} \ar[d]_{\sufi} & \hc 1\diffgrbcon A M \ar[d]^{\sf} \\   h\fusbunth A {LM} \ar[r]_{h\un_x} & \hc 1\diffgrb A M }
\end{equation*}
of categories and functors is strictly commutative. 
\end{maintheorem}

Theorem \ref{th:mainB} holds basically by construction of the regression functor $h\un_x$ (Proposition \ref{prop:comm}). The proof of Theorem \ref{th:mainA} is carried out in Section \ref{sec:proof}. There we prove that the functor $h\un_x$ is a bijection on isomorphism classes of objects, and that it induces on the level of morphisms  equivariant maps between torsors over isomorphic groups. The hardest part is the proof that the functor $h\un_{x}$ is injective on isomorphism classes of objects. One ingredient we use is a generalization of a lemma of Murray \cite{murray} about differential forms on fibre products of a surjective submersion from the \emph{smooth} to the \emph{diffeological} setting -- this is discussed in Appendix \ref{sec:app}.

In Appendix \ref{sec:6} we explain -- separated from the main text -- a byproduct of the discussion in Section \ref{sec:proof}, namely that two compatible, symmetrizing thin structures on the \emph{same} fusion bundle are necessarily isomorphic (Proposition \ref{prop:superficialprop}). This may be surprising since it means that -- on the level of \emph{isomorphism classes} -- a thin structure is no additional structure at all. Still, as we argue in Appendix \ref{sec:6}, thin structures are necessary for the correct \emph{categorical structure}.

\subsection{Application to Lifting Problems}

\label{sec:applications}

\def\adjust#1{\!\!\!\begin{tabular}{c}$#1$\end{tabular}\!\!\!}
\def\hat#1{\widehat{#1}}

In this section we describe an application of our results to lifting problems, completing a discussion started in \cite[Section 1.2]{waldorf10} and \cite{waldorf13}. Let
\begin{equation*}
\alxydim{}{1 \ar[r] & A \ar[r] & \adjust{\hat G} \ar[r]^-{p} & G \ar[r] & 1 }
\end{equation*}
be a central extension of Lie groups, with $A$ abelian. Let $E$ be a principal $G$-bundle over $M$. A \emph{$\hat G$-lift} of $E$ is a principal $\hat G$-bundle $F$ over $M$ together with a smooth, fibre-preserving map $f: F \to E$ satisfying $f(e\cdot g) = f(e) \cdot p(g)$ for all $e\in F$ and $g\in \hat G$. A morphism between $\hat G$-lifts $F$ and $F'$ is a bundle morphism $\varphi : F \to F'$ that exchanges the maps to $E$. We denote by $\struc {\hat G}E$ the category of $\hat G$-lifts of $E$. The obstruction against the existence of $\hat G$-lifts can be represented by a bundle gerbe, the \emph{lifting bundle gerbe $\mathcal{G}_E$} \cite{murray}. Namely, there is an  equivalence
of categories\begin{equation}
\label{eq:lift}
\hom (\mathcal{G}_E,\mathcal{I}) \cong \struc {\hat G} E\text{,}
\end{equation} 
where $\mathcal{I}$ is the trivial bundle gerbe and $\hom$ denotes the Hom-category of the 2-category of  bundle gerbes. In other words, the trivializations of the lifting bundle gerbe $\mathcal{G}_E$ are exactly the $\hat G$-lifts of $E$. 

Since the regression functor $h\un_x$ is essentially surjective by Theorem \ref{th:mainA}, there exists a thin fusion bundle $P_E$ over $LM$ together with a 1-morphism $\mathcal{A}: \mathcal{G}_E \to \un_x(P_E)$. Further, if $\trivlin$ denotes  the trivial thin fusion bundle over $LM$, there is an obvious canonical trivialization $\mathcal{T}: \un_x(\trivlin) \to \mathcal{I}$. Since the regression functor $h\un_x$ is full and faithful, we obtain a bijection
\begin{equation}
\label{eq:bijhom}
\hom(P_E,\trivlin) \to \hc  0 \hom(\mathcal{G}_E,\mathcal{I}): \varphi \mapsto \mathcal{T}  \circ  h\un_x(\varphi) \circ  \mathcal{A}\text{,}
\end{equation}
where $\hom(P_E,\trivlin)$ is the Hom-set in the homotopy category of thin fusion bundles.
The elements in this Hom-set can be identified with homotopy classes of \emph{fusion-preserving, thin sections}, i.e. sections $\sigma: LM \to P_E$ such that
\begin{equation*}
\lambda(\sigma(\gamma_1,\gamma_2) \otimes \sigma(\gamma_2,\gamma_3) ) = \sigma(l(\gamma_1,\gamma_3)) \quand
d_{\tau_0,\tau_1}(\sigma(\tau_0)) = \sigma(\tau_1)
\end{equation*} 
for all $(\gamma_1,\gamma_2,\gamma_3) \in PM^{[3]}$ and all thin homotopic loops $\tau_0,\tau_1 \in LM$. 
Now, the bijection \erf{eq:bijhom} and the equivalence \erf{eq:lift}  together imply the following.

\begin{maintheorem}{C}
\label{th:lifting}
Let $E$ be a principal $G$-bundle over a connected smooth manifold $M$, and let $\hat G$ be a central extension of $G$ by an abelian Lie group $A$. Then, any choice $(P_E,\mathcal{A})$ as above determines a bijection
\begin{equation*}
\bigset{3.5cm}{Isomorphism classes of $\hat G$-lifts of $E$}
\cong
\bigset{3.4cm}{Homotopy classes of fusion-preserving thin sections of $P_E$}\text{.} \end{equation*}
In particular, $E$ admits a $\hat G$-lift if and only if $P_E$ admits a fusion-preserving thin section.
\end{maintheorem}
  
Theorem   \ref{th:lifting} is a complete loop space formulation of lifting problems. It generalizes \cite[Theorem B]{waldorf10} from \emph{discrete} abelian groups $A$ to arbitrary ones. Indeed, if $A$ is discrete, there are no non-trivial homotopies and every section is thin (with respect to the unique trivial thin structure). 

In the discrete case, Theorem   \ref{th:lifting} can be applied to spin structures  on an oriented Riemannian manifold $M$ of dimension $n$, since spin structures are   $\spin n$-lifts of the frame bundle $F M$ of $M$; see \cite[Corollary E]{waldorf10}. For a non-discrete example, we may apply it to $\spinc n$-structures on $M$, which are lifts of $FM$ along the central extension
\begin{equation*}
1 \to \ueins \to \spinc n \to \so n \to 1\text{.}
\end{equation*}
Using the Levi-Civita connection on $M$ and the transgression functor $\trcon$ one can canonically construct the principal $\ueins$-bundle $P_{FM}$ and the isomorphism $\mathcal{A}$ \cite[Section 6]{waldorf13}, slightly improving the general situation. The canonical bundle $P_{FM}$ is also called the \emph{complex orientation bundle} over $LM$, and its sections are called \emph{complex orientations of $LM$}. Then, Theorem \ref{th:lifting} becomes the following.

\begin{maincorollary}{A}
Let $M$ be an oriented Riemannian manifold of dimension $n$. Then, there is a bijection
\begin{equation*}
\bigset{4.4cm}{Isomorphism classes of $\spinc n$-structures on $M$}
\cong
\bigset{4.8cm}{Homotopy classes of fusion-preserving, thin, complex orientations of $LM$}\text{.}
\end{equation*}
\end{maincorollary}

In order to make some more connections to my paper \cite{waldorf13} let us return the  general situation of a central extension $p:\hat G \to G$ of a Lie group $G$ by an abelian Lie group $A$. We recall that if a principal $G$-bundle $E$ carries a connection $\omega \in \Omega^1(E,\mathfrak{g})$, a \emph{geometric $\hat G$-lift} of $E$ is a $\hat G$-lift $F$ together with a connection $\hat\omega \in \Omega^1(F,\hat{\mathfrak{g}})$ such that $f^{*}\omega = p_{*}(\hat\omega)$. Here, $\mathfrak{g}$ and $\hat{\mathfrak{g}}$ are the Lie algebras of $G$ and $\hat G$, respectively, and $\mathfrak{a}$ will denote the Lie algebra of $A$. A morphism between geometric $\hat G$-lifts $F$ and $F'$ is a connection-preserving bundle morphism $\varphi:F \to F'$ that exchanges the maps to $E$. One can associate to each geometric  $\hat G$-lift $F$ a \quot{scalar curvature} $\rho \in \Omega^2(M,\mathfrak{a})$ \cite[Section 2]{waldorf13}. Geometric $\hat G$-lifts of $E$ with fixed scalar curvature $\rho$ form a category  
$\struccon {\hat G}{\rho}E$. The obstruction theory with the lifting bundle gerbe $\mathcal{G}_E$ extends to a setting with connections: one can equip the lifting gerbe $\mathcal{G}_E$ with a connection \cite{gomi3}, such that there is an equivalence of categories 
\begin{equation}
\label{eq:liftcon}
\homcon(\mathcal{G}_E,\mathcal{I}_{-\rho}) \cong \struccon {\hat G}{\rho} E\text{,}
\end{equation} 
see \cite[Theorem 2.2]{waldorf13}.
Here, $\mathcal{I}_{-\rho}$ is the trivial bundle gerbe equipped with the connection given by $-\rho$, and $\homcon$ denotes the Hom-category of the 2-category of bundle gerbes \emph{with connection}. The analogue of Theorem \ref{th:lifting} in the setting with connections is \cite[Theorem A]{waldorf13}: it combines the equivalence \erf{eq:liftcon} with \cite[Theorem A]{waldorf10} and proves that isomorphism classes of \emph{geometric $\hat G$-lifts of $E$ with fixed scalar curvature $\rho$} are in bijection to fusion-preserving sections of $P_E$ of curvature $L\rho$, i.e. sections that pullback the connection 1-form of $P_E$ to the transgressed 1-form $-L\rho$ on $LM$. Under the functor $\sufi$, such sections become thin (Proposition \ref{prop:sftrivial}). Due to Theorem \ref{th:mainB}, we obtain the following.

\begin{maintheorem}{D}
\label{th:liftingcompl}
Let $E$ be a principal $G$-bundle with connection over a connected smooth manifold $M$, and let $\hat G$ be a central extension of $G$ by an abelian Lie group $A$. Let $P_E$ be the transgression of the lifting gerbe $\mathcal{G}_E$, and let $\rho \in \Omega^2(M,\mathfrak{a})$. Then, the diagram
\begin{equation*}
\alxydim{@C=3cm@R=1.5cm}{\hc 0 (\struccon {\hat G} \rho E) \ar[d] \ar@{<->}[r]^-{\cong} & \left\lbrace \begin{minipage}[c]{3cm}
\begin{center}
\baselineskip=1.2em
Fusion-preserving sections of $P_E$ of curvature $-L\rho$
\end{center}
\end{minipage}\right\rbrace \ar[d]^{\sufi} \\ \hc 0 (\struc {\hat G}E) \ar@{<->}[r]_-{\cong} & \left\lbrace
\begin{minipage}[c]{4cm}
\begin{center}
\baselineskip=1.2em
Homotopy classes of fusion-preserving, thin sections of $\sufi(P_E)$
\end{center}
\end{minipage}\right\rbrace }
\end{equation*}
is commutative, and the horizontal arrows are bijections.
\end{maintheorem}

Theorem \ref{th:liftingcompl} is a complete loop space formulation of lifting problems and geometric lifting problems for central extensions by abelian Lie groups.

\setsecnumdepth{2}

\section{Thin Fusion Bundles over Loop Spaces}

\label{sec:fusbun}

In this section we introduce the groupoid $\fusbunth A {LX}$ of thin fusion bundles over the loop space $LX$ of a diffeological space $X$. For an introduction to diffeological spaces we refer to Appendix A of \cite{waldorf9}. The loop space  $LX$ is the diffeological space of smooth maps $\tau:S^1 \to X$. In the following we denote -- for any diffeological space $Y$ -- by $PY$ the diffeological space of paths in $Y$: smooth maps $\gamma: [0,1] \to Y$ with \quot{sitting instants}, i.e. they are locally constant in a neighbourhood of $\left \lbrace 0,1 \right \rbrace$.

\subsection{Thin Bundles}

\label{sec:thin}

One of the main inventions of this article is  a new version of loop space: the \emph{thin loop stack} $\thingrpd LX$. It is based on the notion of a \emph{thin path} in $LX$, also known as \emph{thin homotopy}. A path $\gamma \in PLX$ is called \emph{thin}, if the associated map $\exd\gamma: [0,1] \times S^1 \to X$ is of rank one. If $X$ is a smooth manifold, this means that the rank of the differential of $\exd\gamma$ is at most one at all points. The notion of  rank of a  smooth map between general diffeological spaces is a consistent generalization introduced in \cite[Definition 2.1.1]{waldorf9}. 

The thin loop stack $\mathfrak{L}X$ is not a diffeological \emph{space} but a diffeological \emph{groupoid}.  Its space of objects is the loop space $LX$. We denote by $\thinpairs {LX}$ the set of pairs $(\tau_1,\tau_2)$ of thin homotopic loops  in $X$, and let $\mathrm{pr}_1,\mathrm{pr}_2: \thinpairs {LX}  \to LX$ be the two projections. In order to define a  diffeology on $\thinpairs {LX}$ we have to specify \quot{generalized charts}, called \emph{plots}. Here, we define a map $c:U \to \thinpairs {LX}$ to be a plot if every point $u\in U$ has an open neighbourhood $W$ such that there exists a smooth map $h:W\to PLX$ so that $h(w)$ is a thin path from $\mathrm{pr}_1(c(w))$ to $\mathrm{pr}_2(c(w))$. The three axioms of a diffeology are easy to verify. The diffeological space $\thinpairs {LX}$ is the space of morphisms of the groupoid $\thingrpd LX$. The two projections $\mathrm{pr}_1$ and $\mathrm{pr}_2$ are subductions (the diffeological analogue of a map with smooth local sections) and provide target and source maps of $\thingrpd LX$. Finally, the composition 
\begin{equation*}
\circ: \thinpairs {LX} \;\lli{\mathrm{pr}_2}\times_{\mathrm{pr}_1}\thinpairs {LX} \to \thinpairs {LX}
\end{equation*}
is defined by dropping the loop in the middle. This is smooth since smooth families of thin homotopies can be composed smoothly; see \cite[Proposition 2.1.6]{waldorf9}.

\begin{definition}
\label{def:catalmost}
The category of \emph{almost-thin bundles} over $LX$ is by definition the category $\bun A {\thingrpd LX}$ of  principal $A$-bundles over the thin loop stack $\thingrpd LX$.
\end{definition}

Thus, an \emph{almost-thin bundle} is a principal $A$-bundle over the loop space $LX$ together with a bundle isomorphism
\begin{equation*}
d: \mathrm{pr}_1^{*}P \to \mathrm{pr}_2^{*}P
\end{equation*}
over $\thinpairs {LX}$ satisfying the cocycle condition
\begin{equation*}
d_{\tau_2,\tau_3} \circ d_{\tau_1,\tau_2} = d_{\tau_1,\tau_3}
\end{equation*}
for any triple $(\tau_1,\tau_2,\tau_3)$ of thin homotopic loops. The isomorphism $d$ will be called  \emph{almost-thin structure} on $P$.  
A morphism between almost-thin bundles $(P_1,d_1)$ and $(P_2,d_2)$ is a bundle
morphism $\varphi:P_1 \to P_2$ such that the diagram
\begin{equation*}
\alxydim{@=1.5cm}{\mathrm{pr}_1^{*}P_1 \ar[d]_{d_1} \ar[r]^{\mathrm{pr}_1^{*}\varphi} & \mathrm{pr}_1^{*}P_2 \ar[d]^{d_2} \\ \mathrm{pr}_2^{*}P_1 \ar[r]_{\mathrm{pr}_2^{*}\varphi} & \mathrm{pr}_2^{*}P_2}
\end{equation*}
of bundle morphisms over $\thinpairs{LX}$ is commutative. Such bundle morphisms will be called \emph{thin bundle morphisms}.

Another interesting diffeological groupoid related to loop spaces is the action groupoid $\act {LX} {\diff^{+}(S^1)}$, where $\diff^{+}(S^1)$ denotes the diffeological group of orientation-preserving diffeomorphisms of $S^1$, acting on $LX$ by reparameterization. There is a smooth functor
\begin{equation}
\label{eq:equivstack}
\act {LX}{\diff^{+}(S^1)} \to \thingrpd LX
\end{equation}
which is the identity on the level of objects, and given by
\begin{equation*}
(\tau, \varphi) \mapsto (\tau, \tau \circ \varphi)
\end{equation*}
on the level of morphisms. Indeed, $\varphi$ is homotopic to the identity $\id_{S^1}$, and any such homotopy induces a thin homotopy between $\tau$ and $\tau\circ \varphi$ \cite[Proposition 2.2.5]{waldorf10}. 
Via pullback along the functor \erf{eq:equivstack} we obtain the following.

\begin{proposition}
Every almost-thin structure on a principal $A$-bundle $P$ over $LX$ determines a $\diff^{+}(S^1)$-equivariant structure on $P$. 
\end{proposition}

\begin{remark}
\label{rem:stacks}
The following is a sequence of remarks about further constructions with loop spaces and  comparisons between them.
\begin{enumerate}[(i)]

\item 
\label{rem:stacks:descent}
The first construction is the \emph{thin loop space $\mathcal{L}X$} introduced in \cite{waldorf9}: it is a diffeological space obtained as the quotient of $LX$ by the equivalence relation of thin homotopy. We obtain a diffeological groupoid, denoted by $LX^{[2]}$, with objects $LX$ and morphisms $LX \times_{\mathcal{L}X} LX$, the fibre product of $LX$ with itself over the  projection $\pr: LX \to \mathcal{L}X$. Since diffeological bundles form a stack \cite[Theorem 3.1.5]{waldorf9}, we have an equivalence of categories
\begin{equation}
\label{eq:bundlesequiv}
\bun A {\mathcal{L}X}  \cong \bun A {LX^{[2]}}\text{,}
\end{equation}
i.e., bundles over the groupoid $LX^{[2]}$ are the same as bundles over the thin loop space $\mathcal{L}X$. 

\item
In order to relate the groupoid $LX^{[2]}$ to the other two groupoids discussed so far,
let us consider the identity map
\begin{equation}
\label{eq:identity}
\id: \thinpairs {LX} \to LX \times_{\mathcal{L}X}LX\text{.}
\end{equation}
It is smooth and thus defines a smooth functor $\thingrpd LX \to LX^{[2]}$. A smooth functor is a \emph{weak equivalence}, if it is smoothly essentially surjective, and smoothly fully faithful; see e.g. \cite{lerman1,metzler,Nikolaus}. In that case, the corresponding bundle categories are equivalent \cite[Corollary 2.3.11]{Nikolaus}.
Since the functor is the identity on objects, the first condition is trivial and the second condition is satisfied only if \erf{eq:identity} is a diffeomorphism. 
Invoking the definition of the diffeology on $\thinpairs {LX}$, the assertion that it is a diffeomorphism is equivalent to the following statement: 
\newlength{\meinebreite}
\setlength{\meinebreite}{\textwidth}
\addtolength{\meinebreite}{-8em}
\begin{equation}
\label{eq:assertion}
\begin{minipage}[c]{\meinebreite}
Suppose $c_1$ and $c_2$ are smooth families of loops in $LX$, parameterized by an open set $U \subset \R^n$, such that for each parameter $u \in U$ the loops $c_1(u)$ and $c_2(u)$ are thin homotopy equivalent. Then, one can choose these thin homotopies locally in smooth families. 
\end{minipage}
\end{equation}
I do not know whether or not \erf{eq:assertion} is true; one can certainly choose homotopies in smooth families, but these will generically not be thin. For the time being we must continue with the assumption that the functor $\thingrpd LX \to LX^{[2]}$ is not a weak equivalence, and the bundle categories are not equivalent. 

\item
All together we have a sequence of diffeological groupoids and smooth functors
\begin{equation*}
\act {LX}{\diff^+(S^1)} \to \thingrpd LX \to LX^{[2]}\text{,}
\end{equation*}
and induced pullback functors
\begin{equation*}
\bun A{\mathcal{L}X} \to \bun A {\thingrpd LX} \to \bun A{LX}^{\diff^{+}(S^1)}\text{.}
\end{equation*}
Summarizing, the almost-thin bundles we introduced here have \quot{more symmetry} than $\diff^{+}(S^1)$-equivariance, but not enough to descend to $\mathcal{L}X$.

\item
\label{rem:stacks:quotient}
Let us explain why \cite{waldorf9} only discussed the thin loop space $\mathcal{L}X$ but not the thin loop stack $\thingrpd LX$. There we have considered smooth maps $f: \mathcal{L}X \to A$ defined on the thin loop space. Since maps form a sheaf, such maps are the same as smooth maps on the groupoid $LX^{[2]}$,
\begin{equation}
\label{eq:bijectionmaps1}
D^{\infty}(\mathcal{L}X,A) = D^{\infty}(LX^{[2]},A)\text{.}
\end{equation}
This is just like for bundles; see \erf{eq:bundlesequiv}.

\noindent
The difference is that a map $f: LX^{[2]} \to A$ satisfies a descent \emph{condition} over $LX \times_{\mathcal{L}X} LX$ while a bundle over $LX^{[2]}$ involves a descent \emph{structure} over $LX \times_{\mathcal{L}X} LX$ (a certain bundle morphism).
As mentioned above, the only difference between the groupoids $\thingrpd LX$ and $LX^{[2]}$ lies in the diffeologies on their morphism spaces. A \emph{condition} cannot see this difference, while a \emph{structure} does. Hence, there is a bijection
\begin{equation}
\label{eq:bijectionmaps2}
D^{\infty}(\thingrpd LX,A) \cong D^{\infty}(LX^{[2]},A)\text{,}
\end{equation} 
but no  equivalence between the groupoids of bundles (assuming \erf{eq:assertion} does not hold). Combining \erf{eq:bijectionmaps1} and \erf{eq:bijectionmaps2} we see that for the purpose of \cite{waldorf9} the thin loop space $\mathcal{L}X$ was good enough.

\def\open{\mathcal{O}pen(\R)}

\item

The category of diffeological spaces is equivalent to the category of concrete sheaves over the  site $\mathcal{D}$ of open subsets of $\R^n$, smooth maps, and the usual open covers \cite{baez6}. From this perspective, one can regard the thin loop stack $\thingrpd LX$ as a sheaf of groupoids over $\mathcal{D}$. A construction one can perform now is to take the 0th homotopy sheaf $\tau_0(\thingrpd LX)$, i.e., the coequalizer of source and target. 

The problem is that coequalizers do in general \emph{not} preserve concreteness of sheaves, in which case above construction does not end in a diffeological space; see e.g.
\cite[Remark 1.2.26]{Wu2012}. In the present case one can show that $\tau_0(\thingrpd LX)$ is concrete  only if  \erf{eq:assertion} holds.
In fact, if \erf{eq:assertion} is true, then $\tau_0(\thingrpd LX) = \mathcal{L}X$ as diffeological spaces. 
\end{enumerate}
\end{remark}

An important way to obtain examples of almost-thin bundles over $LX$ is by starting with a bundle with thin connection.

\begin{definition}
\label{def:thin}
A connection on a principal $A$-bundle $P$ over $LX$ is called \emph{thin} if its holonomy around a loop $\tau\in LLX$ vanishes whenever the associated map $\tau^{\vee} \maps S^1 \times S^1 \to X$ is of rank one.
\end{definition}

In other words, if $\gamma_1$ and $\gamma_2$ are thin paths in $LX$ with $\gamma_1(0)=\gamma_2(0)=\tau_0$ and $\gamma_1(1)\eq\gamma_2(1)=\tau_1$, then the parallel transport maps $\tau_{\gamma_1}: P_{\tau_0} \to P_{\tau_1}$ and $\tau_{\gamma_2}: P_{\tau_0} \to P_{\tau_1}$ coincide \cite[Lemma 2.5]{waldorf10}. Hence, a thin connection determines a map $d^{\omega}_{\tau_0,\tau_1}: P_{\tau_0} \to P_{\tau_1}$ between the fibres of $P$ over thin homotopic loops $\tau_1$ and $\tau_2$, \emph{independent} of the choice of the thin path between $\tau_0$ and $\tau_1$.

\begin{lemma}
\label{lem:thinstr}
The maps $d^{\omega}_{\tau_0,\tau_1}$ parameterized by pairs $(\tau_0,\tau_1) \in \thinpairs {LX}$ form an almost-thin structure $d^{\omega}$ on $P$.
\end{lemma}

\begin{proof}
Let $c:U \to \mathrm{pr}_1^{*}P$ be a plot, i.e. the two prolongations $c_P: U \to P$ and $c_b: U \to \thinpairs{LX}$ are smooth. Let $u\in U$. By definition of the diffeology on $\thinpairs {LX}$ the point $u$ has an open neighbourhood $W$ with a smooth map $\gamma: W \to PLX$, such that $\gamma(w)$ is a thin path connecting the loops $\gamma(w)(0)$ and $\gamma(w)(1)$. The parallel transport of the connection $\omega$ can be seen as a smooth map
\begin{equation*}
\tau^{\omega}: W \lli{\ev_0}\times_{p} P \to P\text{,}
\end{equation*}
see \cite[Proposition 3.2.10 (d)]{waldorf9}.
The composite
\begin{equation*}
\alxydim{@C=1.7cm}{W \ar[r]^-{\mathrm{diag}} & W \times W \ar[r]^-{c|_W \times \gamma} & \mathrm{pr}_1^{*}P \times PLX \ar[r]^-{\tau^{\omega}} & \mathrm{pr}_2^{*}P }
\end{equation*}
is a smooth map and coincides with $d^{\omega} \circ c|_W$. Thus, the latter is smooth, and this shows that $d^{\omega}$ is smooth. 
Since parallel transport is functorial under the composition of paths, it is clear that $d^{\omega}$ satisfies the cocycle condition. 
\end{proof}

Almost-thin structures that can be obtained via Lemma \ref{lem:thinstr} play an important role in this article.

\begin{definition}
\label{def:integability}
An almost-thin structure $d$ on a principal $A$-bundle $P$ over $LX$ is called \emph{integrable} or \emph{thin structure}, if there exists a thin connection $\omega$ on $P$ such that $d=d^{\omega}$. 
\end{definition}

Bundles over $LX$ with thin structures form a full subcategory of $\bun A {\thingrpd LX}$ that we denote by $\bunth A {LX}$, i.e. the morphisms in $\bunth A {LX}$  are thin bundle morphisms.

So far we have seen that almost-thin structures can be obtained from thin connections, and that these are (by definition) integrable. Next we relate the notion of a thin structure to the one of a superficial connection \cite[Definition 2.4]{waldorf10}. A connection on a principal $A$-bundle over $LX$ is called \emph{superficial}, if the following two conditions are satisfied:
\begin{enumerate}
\item 
it is thin in the sense of Definition \ref{def:thin};

\item
if $\tau_1,\tau_2\in LLX$ are loops in $LX$ so that their associated maps $\exd{\tau_k}: S^1 \times S^1 \to X$ are rank-two-homotopic, then $\tau_1$ and $\tau_2$ have the same holonomy. 
\end{enumerate}
We denote by $\bunconsf A {LX}$  the groupoid of principal $A$-bundles over $LX$ with superficial connections. 
In particular, since every superficial connection is thin, Lemma \ref{lem:thinstr} defines a functor
\begin{equation}
\label{eq:sufi}
\sufi:\bunconsf A {LX} \to \bunth A {LX}:(P,\omega) \mapsto (P,d^{\omega})\text{.}
\end{equation}

We shall make clear that the functor $\sufi$ loses information. More precisely, we show that \emph{different} superficial connections may define the \emph{same} thin structure.  To start with, we recall that to any $k$-form $\eta \in \Omega^k(X)$ one can associate a $(k-1)$-form $L\eta \in \Omega^{k-1}(LX)$ defined by
\begin{equation*}
L\eta := \int_{S^1} \ev^{*}\eta\text{,}
\end{equation*}
where $\ev: S^1 \times LX \to X$ is the evaluation map. We call a 1-form on $LX$ \emph{superficial}, if it is superficial when considered as a connection on the trivial bundle over $LX$.

\begin{lemma}
Let $\eta \in \Omega^2(X)$. Then, $L\eta$ is superficial. 
\end{lemma}

\begin{proof}
This can be verified via a direct calculation using that the pullback of a $k$-form along a map of rank less than $k$ vanishes (see \cite[Lemma A.3.2]{waldorf9} for a proof of this claim in the diffeological setting). 
\end{proof}

\begin{proposition}
\label{prop:sftrivial}
Suppose $\eta \in \Omega^2(X)$, and consider the 1-form $L\eta$ as a superficial connection on the trivial bundle $\trivlin$ over $LX$.  Then, the induced thin structure $d^{L\eta}$ is the trivial one, i.e. given by the identity bundle morphism $\id: \mathrm{pr}_1^{*}\trivlin \to \mathrm{pr}_2^{*}\trivlin$ over $\thinpairs {LX}$. 
\end{proposition}

\begin{proof}
Let $(\tau_0,\tau_1) \in \thinpairs {LX}$, and let $\gamma \in PLX$ be a thin path connecting $\tau_0$ with $\tau_1$, i.e. the adjoint map $\exd\gamma: [0,1] \times S^1 \to X$ has rank one. The parallel transport in $\trivlin_{L\eta}$ along $\gamma$ is given by multiplication with
\begin{equation*}
\exp \left ( \int_{[0,1]} \gamma^{*}L\eta \right ) = \exp \left (  \int_{[0,1] \times S^1} (\exd\gamma)^{*}\eta \right) =1\text{;}
\end{equation*}
hence the induced thin structure is trivial. 
\end{proof}

In particular, the trivial thin structure on the trivial bundle over $LX$ is induced by \emph{any} 2-form $\eta \in \Omega^2(X)$. This shows  that a lot of structure is lost when passing from a superficial connection to a thin structure.

One suggestion for further research is to find obstructions against the integrability of an almost-thin structure that  makes no direct use of connections  (in contrast to Definition \ref{def:integability}). It might also be possible that there is no obstruction, i.e. that every almost-thin structure is integrable; so far I was unable to prove or disprove this.  

\subsection{Fusion Products}

\label{sec:fus}

We use the smooth maps 
\begin{equation*}
\ev: \p X \to X \times X: \gamma \mapsto (\gamma(0),\gamma(1))
\quomma
l: \p X^{[2]} \to L X: (\gamma_1,\gamma_2) \mapsto \prev{\gamma_2} \pcomp \gamma_1\text{,}
\end{equation*}
where $PX^{[k]}$ denotes the $k$-fold fibre product of $PX$ over $X \times X$  \cite[Section 2]{waldorf9}. Like in Section \ref{sec:maintheorem}, we let $e_{ij}: PX^{[3]} \to LX$ denote the map $l \circ \mathrm{pr}_{ij}$, with $\mathrm{pr}_{ij}:PX^{[3]} \to PX^{[2]}$ the projections. We recall the following.

\begin{definition}[{{\cite[Definition 2.1.3]{waldorf10}}}]
Let $P$ be a principal $A$-bundle over $LX$.
A \emph{fusion product} on $P$ is a bundle morphism
\begin{equation*}
\lambda: e_{12}^{*}P \otimes e_{23}^{*}P \to e_{13}^{*}P
\end{equation*}
over $PX^{[3]}$ that is associative over $PX^{[4]}$. A bundle equipped with a fusion product is called \emph{fusion bundle}. A bundle isomorphism is called \emph{fusion-preserving}, if it commutes with the fusion products. \end{definition}

There are two compatibility conditions between a connection on a principal $A$-bundle $P$ and a fusion product $\lambda$. We say that the connection is \emph{compatible} with $\lambda$ if the isomorphism $\lambda$ preserves connections. We say \cite[Definition 2.9]{waldorf10} that the connection \emph{symmetrizes} $\lambda$, if 
\begin{equation}
\label{eq:symmetrizing}
R_{\pi}(\lambda(q_1 \otimes q_2)) = \lambda(R_{\pi}(q_2) \otimes R_{\pi}(q_1))
\end{equation}
for all $(\gamma_1,\gamma_2,\gamma_3) \in PX$ and all $q_1 \in P_{l(\gamma_1,\gamma_2)}$ and $q_2\in P_{l(\gamma_2,\gamma_3)}$, where $R_{\pi}$ is the parallel transport along the rotation by an angle of $\pi$, regarded as a path in $LX$. The reader may check that \erf{eq:symmetrizing} makes sense. Condition \erf{eq:symmetrizing} can be seen as a weakened  commutativity condition for $\lambda$ \cite[Remark 2.1.6]{waldorf10}.

A \emph{fusion bundle with connection}  is by definition a principal $A$-bundle over $LX$ equipped with a fusion product and a compatible, symmetrizing connection. Morphisms between fusion bundles with connection are fusion-preserving, connection-preserving bundle morphisms. This defines the groupoid $\fusbuncon A {LX}$. Requiring additionally that the connections are superficial defines a full subgroupoid $\fusbunconsf A {LX}$; this groupoid appears prominently  in \cite{waldorf10}.

\begin{definition}
\label{def:compandsym}
Let $(P,\lambda)$ be a fusion bundle. A thin structure $d$ on $P$ is called \uline{compatible and symmetrizing} if there exists an integrating connection that is compatible and symmetrizing.
\end{definition}

A compatible and symmetrizing thin structure $d$ satisfies, in particular, the following two conditions, whose formulation is independent of an integrating connection:
\begin{enumerate}[(a)]
\item 
for all  paths $\Gamma \in P(\px Xx^{[3]})$ with $(\gamma_1,\gamma_2,\gamma_3) := \Gamma(0)$ and $(\gamma_1',\gamma_2',\gamma_3') := \Gamma(1)$ such that the three paths $Pe_{ij}(\Gamma) \in PLX$ are  thin, the diagram
\begin{equation*}
\alxydim{@=1.5cm}{P_{l(\gamma_1,\gamma_2)} \ar[d]_{d \otimes d} \otimes P_{l(\gamma_2,\gamma_3)} \ar[r]^-{\lambda} & P_{l(\gamma_1,\gamma_3)} \ar[d]^{d} \\P_{l(\gamma'_1,\gamma'_2)} \otimes P_{l(\gamma'_2,\gamma'_3)} \ar[r]_-{\lambda} & P_{l(\gamma'_1,\gamma'_3)}}
\end{equation*}
is commutative. 
\item
for all $(\gamma_1,\gamma_2,\gamma_3) \in PX^{[3]}$ and all $q_1\in P_{l(\gamma_1,\gamma_2)}$ and $q_2 \in P_{l(\gamma_2,\gamma_3)}$
\begin{equation*}
\pi^{*}d(\lambda(q_1 \otimes q_2)) = \lambda(\pi^{*}d(q_2) \otimes \pi^{*}d(q_1))\text{,}
\end{equation*}
where $\pi: PX^{[2]} \to \thinpairs {LX}$ is defined by $\pi(\gamma_1,\gamma_2) := (l(\gamma_1,\gamma_2),l(\prev{\gamma_2},\prev{\gamma_1}))$. The relation to \erf{eq:symmetrizing} arises because the two loops in the image of $\pi$ are related by a rotation by an angle of $\pi$. 
\end{enumerate}
Conversely, if an almost-thin structure on a fusion bundle $P$ satisfies (b), and $\omega$ is an integrating thin connection, then $\omega$ is automatically symmetrizing. This follows because the condition of being symmetrizing  only involves the parallel transport of $\omega$ along \emph{thin} paths, and these parallel transports are determined by the almost-thin structure $d$. I do not know whether or not a similar statement holds for (a); at the moment it seems that Definition \ref{def:compandsym}  is  stronger than (a) and (b). 

According to Definition \ref{def:fusbunth}, a \emph{thin fusion bundle} is a fusion bundle over $LX$ with a compatible and symmetrizing thin structure. Thin fusion bundles form a category that we denote by $\fusbunth A {LX}$. Evidently, the functor $\sufi$ from \erf{eq:sufi} passes to the setting with fusion products:
\begin{equation*}
\sufi: \fusbunconsf A {LX} \to \fusbunth A {LX} : (P,\lambda,\omega) \mapsto (P,\lambda,d^{\omega})\text{.}
\end{equation*}

We recall that the main result of  \cite{waldorf10} is that the category $\fusbunconsf A {LM}$ of fusion bundles with superficial connection (over a connected smooth manifold $M$)
is equivalent to the category $\hc 1 \diffgrbcon A M$ of  diffeological bundle gerbes with connection over $M$ via a \emph{transgression functor}
\begin{equation*}
\trcon: \hc 1 \diffgrb A {M} \to \fusbunconsf A{LM}\text{.}
\end{equation*}
 As a preparation for Section \ref{sec:proof} we need the following result about the composition of $\trcon$ with the functor $\sufi$.

Let $\eta \in \Omega^2_{\mathfrak{a}}(M)$ be a 2-form on $M$ with values in the Lie algebra $\mathfrak{a}$ of $A$, and let $\mathcal{I}_{\eta}$ be the trivial gerbe over $M$ with connection $\eta$. There is a fusion-preserving, connection-preserving bundle morphism
\begin{equation*}
\sigma: \trcon(\mathcal{I}_{\eta}) \to \trivlin_{-L\eta}
\end{equation*} 
constructed in \cite[Lemma 3.6]{waldorf13}, where the trivial bundle $\trivlin$ is equipped with the superficial connection 1-form $-L\eta$ and with the trivial fusion product induced by the multiplication in $A$. 

\begin{proposition}
\label{prop:sufitriv}
The bundle morphism $\sigma$ induces a fusion-preserving, thin bundle morphism
\begin{equation*}
(\sufi \circ \trcon)(\mathcal{I}_{\eta}) \cong \trivlin\text{,}
\end{equation*}
where the trivial bundle $\trivlin$ is equipped with the trivial fusion product and the trivial thin structure. 
\end{proposition}

\begin{proof}
By functoriality $\sufi(\trcon(\mathcal{I}_{\eta}))$ is isomorphic to $\sufi(\trivlin_{-L\eta})$ via $\sufi(\sigma)$. By Proposition \ref{prop:sftrivial}, the latter has the trivial thin structure. 
\end{proof}

\subsection{Homotopy Categories of Bundles}

\label{sec:homcat}

It is familiar for topologists to consider the \emph{homotopy category} of topological spaces (or of other model categories). Here we do the same thing with categories of bundles over $LX$.

We have the category $h \bun A {LX}$ of principal $A$-bundles over $LX$ and  homotopy classes of bundle morphisms. Here, a homotopy between two  bundle morphisms $\varphi_0,\varphi_1: P \to P'$ is a smooth map $h: [0,1] \times P \to P'$ so that $h_0 = \varphi_0$ and $h_1 = \varphi_1$, and $h_t$ is a bundle morphism for all $t$. We will often identify the homotopy $h$ with its \quot{difference from $\varphi_0$}, i.e. with the smooth map  $h' \maps [0,1] \nobr\times\nobr LX \to A$ defined by $h_t = \varphi_0 \cdot h'_t$. 

Now we take additional structure into account. We have the category $h\fusbun A {LX}$ of fusion bundles over $LX$ and homotopy classes of fusion-preserving bundle morphisms. Here, the homotopies $h$ satisfy -- in addition to the above -- the condition that $h_t$ is fusion-preserving for all $t$. Correspondingly, the difference maps $h'$ are such that $h'_t$ satisfies the fusion condition for all $t$, i.e.
\begin{equation*}
h_t'(l(\gamma_1,\gamma_2)) \cdot h_t'(l(\gamma_2,\gamma_3)) = h_t'(l(\gamma_1,\gamma_3))
\end{equation*} 
for all triples $(\gamma_1,\gamma_2,\gamma_3) \in PX^{[3]}$.

Next we look at the category $h \bunth A {LX}$ of thin principal $A$-bundles over $LX$ together with homotopy classes of thin bundle morphisms. That is, the homotopies $h$ are such that $h_t$  is thin for all $t$. We claim that the corresponding maps $h'_t$ are such that $h_t'(\tau_0) = h_t'(\tau_1)$ for thin homotopic loops $\tau_0,\tau_1$. 
In other words, $h'$ is a smooth map $h': [0,1] \times \mathcal{L}M \to A$ defined on the \emph{thin} loop space.
In order to prove that claim, suppose that $\varphi_0,\varphi_1: P \to P'$ are thin bundle morphisms, and $h':[0,1] \times LX \to A$ is a smooth map with $h'_0 \equiv 1$ and $\varphi_1 = \varphi_0 \cdot h'_1$. Let $\tau_0$ and $\tau_1$ be thin homotopic loops, and let   $p \in P_{\tau_0}$. Then, we get
\begin{multline*}
d_{\tau_0,\tau_1}(\varphi_0(p))  \cdot h'_t(\tau_0)
= d_{\tau_0,\tau_1}(\varphi_0(p) \cdot h'_t(\tau_0)) 
= d_{\tau_0,\tau_1}(h_t(p))
= h_t(d_{\tau_0,\tau_1}(p))
\\= \varphi_0(d_{\tau_0,\tau_1}(p)) \cdot h_t'(\tau_1)
= d_{\tau_0,\tau_1}(\varphi_0(p)) \cdot h_t'(\tau_1)\text{.}
\end{multline*}

Combining the last two paragraphs, we arrive at Definition \ref{def:hfusbunth} of Section \ref{sec:results}, the homotopy category $h\fusbunth A {LX}$ of thin fusion bundles over $LX$, consisting of thin fusion bundles and  homotopy classes of fusion-preserving, thin bundle morphisms. The maps difference maps $h'$ are smooth maps $h': [0,1] \times \mathcal{L}X \to A$ that are fusion maps for all $t$; such maps have been called \emph{fusion homotopies} in \cite{waldorf9}. The homotopy category $h\fusbunth A {LX}$ plays a central role in this paper, it is the category on the right hand side of Theorem \ref{th:mainA}. For the proof of Theorem \ref{th:mainA} in Section \ref{sec:proof} we provide the following result.

\begin{proposition}
\label{prop:homtorsor}
The Hom-sets in the homotopy category $h\fusbunth A{LX}$ are either empty or torsors over the group $h \fus A{\mathcal{L}X}$ of homotopy classes of fusion maps.
\end{proposition}

\begin{proof}
Let $f: \mathcal{L}X \to A$ be a fusion map,  let $(P_1,\lambda_1,d_1)$ and $(P_2,\lambda_2,d_2)$ be thin fusion bundles, and let $\varphi:P_1 \to P_2$ be a thin, fusion-preserving bundle morphism. We define 
\begin{equation*}
(\varphi \cdot f)(p_1) := \varphi(p_1)\cdot f(\tau)
\end{equation*}
for all $p_1 \in P_1|_{\tau}$ and prove that this gives the claimed free and transitive action. First of all, it is easy to check that the action is well-defined, i.e. that the homotopy class of $\varphi \cdot f$ depends only on the homotopy classes of $\varphi$ and $f$, 
and that $\varphi \cdot f$ is a fusion-preserving, thin bundle morphism. 

The action is free: suppose $\varphi$ is a thin, fusion-preserving bundle morphism, $f$ is a fusion map, and  $h': [0,1] \times \mathcal{L}X \to A$ parameterizes a homotopy between $\varphi \cdot f$ and $\varphi$. But then, the same $h'$ is a fusion homotopy between $f$ and the constant map $1$, so that $f=1$ in $\fushom A{\mathcal{L}X}$. 

The action is transitive: suppose $\varphi,\varphi': P \to P'$ are fusion-preserving, thin bundle morphisms. General bundle theory \cite[Lemma 3.1.3]{waldorf9} provides a unique smooth map $f: LX \to A$ such that $\varphi' = \varphi \cdot f$. That $\varphi$ and $\varphi'$ are fusion-preserving implies that $f$ is a fusion map. That $\varphi$ and $\varphi'$ are thin implies that $f$ is constant on thin homotopy classes of loops:
\begin{multline*}
d'_{\tau_0,\tau_1}(\varphi(p))  \cdot f(\tau_0)
= d'_{\tau_0,\tau_1}(\varphi(p) \cdot f(\tau_0)) 
= d'_{\tau_0,\tau_1}(\varphi'(p))
= \varphi'(d_{\tau_0,\tau_1}(p))
\\= \varphi(d_{\tau_0,\tau_1}(p)) \cdot f(\tau_1)
= d'_{\tau_0,\tau_1}(\varphi(p)) \cdot f(\tau_1)\text{.}
\end{multline*}
Thus, $f$ is a fusion map. 
\end{proof}

\setsecnumdepth{1}

\section{Regression}

\label{sec:regression}

The main ingredient for regression 
are fusion products. In \cite{waldorf10} we have defined a functor
\begin{equation*}
\un_x: \fusbun A{L X} \to \hc 1 \diffgrb A X\text{,}
\end{equation*}
where $x\in X$ is a base point, and $\diffgrb A X$ is the 2-category of diffeological $A$-bundle gerbes over $X$.  This 2-category is discussed in \cite[Section 3.1]{waldorf10}.
Let us  review the construction of the functor $\un_x$. If $(P,\lambda)$ is a fusion bundle over $LX$, the diffeological bundle gerbe $\un_x(P,\lambda)$ consists of the subduction $\ev_1: \px Xx \to X$, the principal $A$-bundle $l^{*}P$ obtained by pullback along $l: \px Xx^{[2]} \to LX$, and the fusion product $\lambda$ as its bundle gerbe product. 
If $\varphi: P_1 \to P_2$ is a fusion-preserving bundle morphism, a 1-isomorphism \begin{equation*}
\un_x(\varphi): \un_x(P_1,\lambda_1) \to \un_x(P_2,\lambda_2)
\end{equation*}
is constructed as follows. Its principal $A$-bundle $Q$ over $\px Xx^{[2]}$ is $l^{*}P_2$, and its bundle isomorphism 
\begin{equation*}
\alpha\maps \mathrm{pr}_{13}^{*}l^{*}P_1 \otimes \mathrm{pr}_{34}^{*}Q \to \mathrm{pr}_{12}^{*}Q \otimes \mathrm{pr}_{24}^{*}l^{*}P_2
\end{equation*}
over $\px Xx^{[4]}$, where $\mathrm{pr}_{ij}: \px Xx^{[4]} \to \px Xx^{[2]}$ denotes the projections, is given by
\begin{equation*}
\alxydim{@=1.3cm}{e_{13}^{*}P_1 \otimes e_{34}^{*}P_2 \ar[r]^-{e_{13}^{*}\varphi \otimes \id} & e_{13}^{*}P_2 \otimes e_{34}^{*}P_2 \ar[r]^-{\mathrm{pr}_{134}^{*}\lambda_2} & e_{14}^{*}P_2 \ar[r]^-{\mathrm{pr}_{124}^{*}\lambda_2^{-1}} & e_{12}^{*}P_2 \otimes e_{24}^{*}P_2\text{,}}
\end{equation*}
where we have used the notation $e_{ij} := l \circ \mathrm{pr}_{ij}$. The compatibility condition between $\alpha$ and the bundle gerbe products $\lambda_1$ and $\lambda_2$ follows from the fact that $\varphi$ is fusion-preserving.

It is straightforward to see that the functor $\un_x$ is monoidal. In particular, if $\trivlin$ denotes the trivial bundle with the trivial fusion product, i.e. the tensor unit in $\fusbun A{LX}$, there is a canonical trivialization $\mathcal{T}: \un_x(\trivlin) \to \mathcal{I}$, where $\mathcal{I}$ denotes the trivial bundle gerbe. Further, the functor $h\un_x$ is natural with respect to base point-preserving smooth maps between diffeological spaces.

Our first aim is to prove that the functor $\un_x$ factors through the homotopy category of fusion bundles; for that purpose we prove the following lemma. 

\begin{lemma}
\label{lem:regequiv}
Let $\varphi: P_1 \to P_2$ be a fusion-preserving morphism between fusion bundles over $LX$, and let $f: \mathcal{L}X \to A$ be a fusion map. Then, there exists a 2-isomorphism
\begin{equation*}
\beta_{f,\varphi}: \un_x(\varphi\cdot f) \Rightarrow \un_x(\varphi) \otimes \un_x(f)\text{,}
\end{equation*}
where $\un_x(f)$ is the principal $A$-bundle over $X$ reconstructed from $f$, and $\otimes$ denotes the action of bundles on 1-isomorphisms between gerbes. 
\end{lemma}

\begin{proof}
We recall from \cite[Section 4.1]{waldorf9} that the bundle $\un_x(f)$ is obtained by descent theory: it descends from the trivial principal $A$-bundle $\trivlin$ over the space $\ptx Xx$ of thin homotopy classes of paths starting at $x$ using a descent structure defined by $f$. In other words, the pullback $P_f := \un_x(f)$ along $\ev: \ptx Xx \to X$ has a canonical trivialization $t: \trivlin \to P_f$, and the diagram
\begin{equation*}
\alxydim{@=1.5cm}{\trivlin \ar[r]^{f}\ar[d]_{\mathrm{pr}_2^{*}t} & \trivlin\ar[d]^{\mathrm{pr}_1^{*}t} \\ \mathrm{pr}_2^{*}P_f \ar@{=}[r] & \mathrm{pr}_1^{*}P_f}
\end{equation*}
over $\ptx Xx^{[2]}$ is commutative, in which the map $f$ is considered as a morphism between trivial bundles. We obtain another commutative diagram
\begin{equation}
\label{eq:2isocomm}
\alxydim{@C=1.7cm@R=1.5cm}{l^{*}P_1 \ar[r]^{l^{*}\varphi\cdot f}\ar[d]_{\mathrm{pr}_2^{*}t} & l^{*}P_2\ar[d]^{\mathrm{pr}_1^{*}t} \\ l^{*}P_1 \otimes \mathrm{pr}_2^{*}P_f \ar[r]_-{l^{*}\varphi \otimes \id} & \mathrm{pr}_1^{*}P_f \otimes l^{*}P_2\text{.}}
\end{equation}
Now we are ready to construct the 2-isomorphism $\beta_{f,\varphi}$. We recall that a 2-isomorphism is an isomorphism between the principal $A$-bundles of the two involved 1-isomorphisms, satisfying a certain condition. In our case the bundles of $\un_x(\varphi \cdot f)$ and $\un_x(\varphi) \otimes P_f$ are $l^{*}P_2$ and $l^{*}P_2 \otimes \mathrm{pr}_1^{*}P_f$, respectively, living over $\px Xx^{[2]}$.  The isomorphism $\beta$ is given by 
\begin{equation*}
\alxydim{@C=2cm}{l^{*}P_2 = l^{*}P_2 \otimes \trivlin \ar[r]^{\id \otimes \mathrm{pr}_1^{*}t} & l^{*}P_2 \otimes  \mathrm{pr}_1^{*}P_f} \text{.}
\end{equation*} 
The condition $\beta$ has to satisfy in order to define the desired 2-isomorphism is  that the diagram
\begin{equation*}
\tiny
\alxydim{@C=1.4cm@R=1.5cm}{e_{13}^{*}P_1 \otimes e_{34}^{*}P_2 \ar[d]_{\id \otimes \mathrm{pr}_{34}^{*}\beta} \ar[r]^-{e_{13}^{*}(\varphi \cdot f) \otimes \id} & e_{13}^{*}P_2 \otimes e_{34}^{*}P_2 \ar[d]^{\mathrm{pr}_{13}^{*}\beta \otimes \id} \ar[r]^-{\mathrm{pr}_{134}^{*}\lambda_2} & e_{14}^{*}P_2 \ar[d]_{\mathrm{pr}_{14}^{*}\beta} \ar[r]^-{\mathrm{pr}_{124}^{*}\lambda_2^{-1}} & e_{12}^{*}P_2 \otimes e_{24}^{*}P_2 \ar[d]^{\mathrm{pr}_{12}^{*}\beta \otimes \id}
\\
e_{13}^{*}P_1 \otimes e_{34}^{*}P_2 \otimes \mathrm{pr}_3^{*}P_f \ar[r]_{e_{13}^{*}\varphi \otimes \mathrm{swap}} & e_{13}^{*}P_2 \otimes \mathrm{pr}_1^{*}P_f \otimes e_{34}^{*}P_2  \ar[r]_-{\mathrm{pr}_{134}^{*}\lambda_2} & \mathrm{pr}_1^{*}P_f \otimes  e_{14}^{*}P_2 \ar[r]_-{\mathrm{pr}_{124}^{*}\lambda_2^{-1}} & e_{12}^{*}P_2  \otimes \mathrm{pr}_1^{*}P_f \otimes e_{24}^{*}P_2}
\end{equation*}
is commutative. Indeed, the subdiagram on the left is commutative due to the commutativity of diagram \erf{eq:2isocomm}, while the other two subdiagrams commute due to the definition of $\beta$.
\end{proof}

Now we look at the functor
\begin{equation}
\label{eq:funhomfact}
\alxydim{@R=1.5cm}{\fusbunth A {LX} \ar[r]  & \fusbun A{LX} \ar[d]^{\un_x} \\   &  \hc 1 \diffgrb AX\text{,}}
\end{equation}
where the horizontal functor simply forgets the thin structures.

\begin{proposition}
\label{prop:regfact}
Over a connected smooth manifold $M$, 
the functor \erf{eq:funhomfact} factors uniquely through the homotopy category $h\fusbunth A {LM}$, i.e. there is a unique functor
\begin{equation*}
h\un_x:h \fusbunth A{L M} \to \hc 1\diffgrb A M
\end{equation*}
such that the diagram
\begin{equation}
\label{eq:funhomfact2}
\alxydim{@R=1.5cm}{\fusbunth A {LM} \ar[r]^-{} \ar[d] & \fusbun A {LM} \ar[d]^{\un_x} \\ h\fusbunth A{LM} \ar[r]_-{h\un_x} &  \hc 1 \diffgrb AM }
\end{equation}
is strictly commutative.
\end{proposition}

\begin{proof}
Uniqueness is clear because the projection to the homotopy category is surjective on objects and morphisms. Suppose $\varphi_0,\varphi_1: P \to P'$ are homotopic morphisms between thin fusion bundles. We have to construct a 2-isomorphism $\un_x(\varphi_0) \cong \un_x(\varphi_1)$, identifying the two 1-morphisms in $\hc 1 \diffgrb A M$. As explained in Section \ref{sec:homcat}, the homotopy is a  smooth map $h': [0,1] \times \mathcal{L}X \to A$ such that $h'_0=1$ and $\varphi_1 = \varphi_0 \cdot h'_1$. In that situation, Lemma \ref{lem:regequiv}  provides a 2-isomorphism
\begin{equation*}
\beta: \un_x(\varphi_1) \Rightarrow \un_x(\varphi_{0}) \otimes \un_x(h'_{1})\text{.}
\end{equation*}
Further, since $h'$ is a fusion homotopy from $h_1'$ to $1$, the principal $A$-bundle $\un_x(h_1')$ is trivializable \cite[Lemma 4.1.3]{waldorf9}, this argument requires the restriction to a smooth manifold.  Under such a trivialization, the 2-isomorphism $\beta$ is a 2-isomorphism $\un_x(\varphi_1) \nobr\cong\nobr \un_x(\varphi_0)$. \end{proof}

We recall from \cite[Section 5.2]{waldorf10} that if the fusion bundles are additionally equipped with superficial connections, the regression functor $\un_x$ can be refined to another functor \begin{equation*}
\uncon_x: \fusbunconsf A {LM} \to \hc 1 \diffgrbcon AM\text{.}
\end{equation*}
The two regression functors $\uncon_x$ and $h\un_x$ are compatible in the following sense.\begin{proposition}
\label{prop:comm}
Let $M$ be a connected smooth manifold. Then, the diagram
\begin{equation*}
\alxydim{@C=2cm@R=1.5cm}{\fusbunconsf A {LM} \ar[d]_{\sufi} \ar[r]^-{\uncon_x} &   \hc 1\diffgrbcon A M \ar[d]^{} \\ h\fusbunth A {LM} \ar[r]_-{h\un_x} & \hc 1 \diffgrb A M  }
\end{equation*}
is strictly commutative.
\end{proposition}

\begin{proof}
The commutativity follows from the one of the following subdiagrams:
\begin{equation*}
\alxydim{@C=1cm}{\fusbunconsf A {LM} \ar[dd]_{\sufi}  \ar[drr] \ar[dr]_{} \ar[rrr]^-{\uncon_x} &&&  \hc 1 \diffgrbcon A M \ar[dd]^{} \\ & \fusbunth A {LM} \ar[r] \ar[dl] & \fusbun A{LM} \ar[dr]^{\un_x} \\  h\fusbunth A {LM} \ar[rrr]_-{h\un_x} &&&  \hc 1\diffgrb A M \text{.} }
\end{equation*}
Indeed, the two triangular diagrams commute evidently, the one on the bottom commutes due to Proposition \ref{prop:regfact}, and the big one commutes because $\uncon_x$ just constructs a gerbe connection on the bundle gerbe constructed by $\un_x$. 
\end{proof}

Finally, we need in Section \ref{sec:proof} the following reformulation of Lemma \ref{lem:regequiv}.

\begin{proposition}
\label{prop:equiv}
Let $P_1$ and $P_2$ be thin fusion bundles over $LM$.
The regression functor $h\un_x$ is equivariant on Hom-sets in the sense that the diagram
\begin{equation*}
\alxydim{@C=1cm@R=1.5cm}{\hom(P_1,P_2) \times h\fus A{\mathcal{L}M}   \ar[d]_-{h\un_x \times \un_x}  \ar[r] & \hom(P_1,P_2) \ar[d]^{h\un_x}    \\  \hom(h\un_x(P_1),h\un_x(P_2)) \times \hc 0 \diffbun AM \ar[r] & \hom(h\un_x(P_1),h\un_x(P_2)) }
\end{equation*}
is commutative. Here, the Hom-sets are those of the categories $h\fusbunth A{LM}$ and $\hc 1 \diffgrb AM$, respectively, and the horizontal maps are the actions of Proposition \ref{prop:homtorsor},  and the usual action of bundles on isomorphisms between gerbes, respectively. 
\end{proposition}

\setsecnumdepth{1}

\section{Proof of the Main Theorem}

\label{sec:proof}

The main theorem (the regression functor $h\un_x$ is an equivalence of categories) is proved by Proposition \ref{prop:esssurj} (it is essentially surjective) and Proposition \ref{prop:fullf} (it is full and faithful) below. For the first part we need the following.

\begin{lemma}
\label{lem:forgetess}
Let $M$ be a smooth manifold.
The functor 
\begin{equation*}
f: \diffgrbcon A M \to \diffgrb AM
\end{equation*}
that forgets the connections
is essentially surjective. \end{lemma}

\begin{proof}
Every diffeological bundle gerbe $\mathcal{G}$ is isomorphic to a smooth one, because its subduction can be refined to a surjective submersion \cite[Lemma A.2.2]{waldorf9}, in which case its principal $A$-bundle becomes a smooth one \cite[Theorem 3.1.7]{waldorf9}. But every smooth bundle gerbe admits a connection \cite{murray}. 
\end{proof}

\begin{proposition}
\label{prop:esssurj}
For $M$ a smooth manifold, the functor $h\un_x$ is essentially surjective. 
\end{proposition}

\begin{proof}
Suppose $\mathcal{G}$ is a diffeological bundle gerbe over $X$. By Lemma \ref{lem:forgetess} there exists a smooth bundle gerbe $\mathcal{H}$ with connection such that $f(\mathcal{H}) \cong \mathcal{G}$, and thus a fusion bundle with superficial connection $P := \trcon(\mathcal{H})$ over $LM$.  We claim that $\sufi(P)$ is an essential preimage for $\mathcal{G}$. Indeed, we have
\begin{equation*}
h\un_x(\sufi(P)) = f(\uncon_x(P)) =f(\uncon_x(\trcon(\mathcal{H}))) \cong f(\mathcal{H}) \cong \mathcal{G}\text{,}
\end{equation*}
using Proposition \ref{prop:comm} and \cite[Theorem A]{waldorf10}.
\end{proof}

Next we want to prove that the functor $h\un_x$ is full and faithful. Our strategy is to prove first that it is essentially injective, i.e. injective on isomorphism classes. This is in fact the hardest part; the remaining part is then a mere consequence of    Proposition \ref{prop:equiv}. Essential injectivity is proved in the following three lemmata. The first one is well-known for \emph{smooth} bundle gerbes, but requires special attention for \emph{diffeological} ones. 

\begin{lemma}
\label{lem:connmorph}
Suppose $\mathcal{G}$ and $\mathcal{H}$ are diffeological bundle gerbes over a smooth manifold $M$, and $\mathcal{A}: \mathcal{G} \to \mathcal{H}$ is an isomorphism. Suppose $\mathcal{G}$ and $\mathcal{H}$ are equipped with connections. Then, there exists a connection on the isomorphism $\mathcal{A}$ and a 2-form $\eta \in \Omega^2_{\mathfrak{a}}(M)$ such that
\begin{equation*}
\mathcal{A}: \mathcal{G} \to \mathcal{H} \otimes \mathcal{I}_{\eta}
\end{equation*}
is connection-preserving. 
\end{lemma}

\begin{proof}
Suppose the 1-isomorphism $\mathcal{A}$ has a refinement $Z$ of the surjective submersions $Y_{\mathcal{G}}$ and $Y_{\mathcal{H}}$, and a principal $A$-bundle $Q$ over $Z$. By refining the refinement $Z$ along the bundle projection $Q \to Z$, we get a new 1-isomorphism with the trivial bundle $\trivlin$ over $Q$, consisting further of an isomorphism
\begin{equation*}
\alpha: P_\mathcal{G} \to P_\mathcal{H}
\end{equation*}  
of principal $A$-bundles over $Q^{[2]} = Q \times_M Q$. It is compatible with the bundle gerbe products $\mu_{\mathcal{G}}$ and $\mu_{\mathcal{H}}$. In general it does not preserve the given connections on $P_{\mathcal{G}}$ and $P_{\mathcal{H}}$. The deviation from being connection-preserving is a 1-form $\omega \in \Omega^1(Q^{[2]},\mathfrak{a})$. It satisfies the cocycle condition
\begin{equation*}
\mathrm{pr}_{13}^{*}\omega = \mathrm{pr}_{12}^{*}\omega + \mathrm{pr}_{23}^{*}\omega
\end{equation*} 
over $Q^{[3]}$. Since $Q \to M$ is a subduction over a smooth manifold, Lemma \ref{lem:exactsequence} implies that there exists a 1-form $\gamma \in \Omega^1_{\mathfrak{a}}(Q)$ such that $\omega = \mathrm{pr}_2^{*}\gamma - \mathrm{pr}_{1}^{*}\gamma$. Equipping the trivial bundle $\trivlin$ over $Q$ with the connection 1-form $\gamma$, the isomorphism
\begin{equation*}
\alpha: P_{\mathcal{G}} \otimes \mathrm{pr}_2^{*}\trivlin_\gamma \to \mathrm{pr}_1^{*}\trivlin_\gamma \otimes P_{\mathcal{H}}
\end{equation*}
is connection-preserving. Now, the 2-form
$C_{\mathcal{H}}-C_{\mathcal{G}} + \mathrm{d}\gamma \in \Omega_{\mathfrak{a}}^2(Q)$ descends to the desired 2-form $\eta \in \Omega^2_{\mathfrak{a}}(M)$.
\end{proof}

The next lemma states that every thin fusion bundle is isomorphic to another one obtained by transgression. For the purpose of this section, let us fix the following terminology.

\begin{definition}
Suppose $(P,\lambda)$ is a fusion bundle with compatible and symmetrizing thin structure $d$. A connection on the bundle gerbe $\un_x(P,\lambda)$, consisting of a connection $\omega$ on the bundle $l^{*}P$ and a curving $B \in \Omega^2(\px Mx)$, is called \emph{good} if there exists a compatible and symmetrizing thin connection $\omega'$ on $P$ with $d=d^{\omega'}$ and  $l^{*}\omega' = \omega$. \end{definition}

Good connections always exist: by definition of a thin fusion bundle there exists a compatible and symmetrizing thin connection $\omega'$ on $(P,\lambda)$ with $d^{\omega'}=d$. Then define $\omega := l^{*}\omega'$. Further, with Lemma \ref{lem:exactsequence} the usual argument for the existence of curvings applies: $\mathrm{d}\omega \in \Omega^2(\px Mx^{[2]})$ satisfies $\delta(\mathrm{d}\omega) = 0$, so that there exists a 2-form $B \in \Omega^2(\px Mx)$ satisfying the condition for curvings, namely $\delta B = \mathrm{d}\omega$. Thus, the pair $(B,\omega)$ is a good connection.

\begin{lemma}
\label{lem:transreg}
Let $(P,\lambda,d)$ be a thin fusion bundle,  let $(B,\omega)$ be a good connection on $\un_x(P,\lambda)$, and write $\mathcal{G}$ for the corresponding bundle gerbe with connection.  Then there exists a fusion-preserving, thin bundle isomorphism 
\begin{equation*}
\sufi(\trcon_{\mathcal{G}}) \cong (P,\lambda,d)\text{.}
\end{equation*}
\end{lemma}

\begin{proof}
The data $(P,\lambda)$ and $\omega$ satisfies the assumptions of \cite[Lemmata 6.2.1, 6.2.2]{waldorf10}, so that there exists a well-defined, fusion-preserving smooth bundle isomorphism
\begin{equation*}
\varphi: \mathscr{T}_{\mathcal{G}}^{\nabla} \to P\text{.}
\end{equation*}
Let $\gamma \in PLM$ be a thin path and let $\exd\gamma: [0,1] \times S^1 \to M$  be the associated rank one map.  Let $\mathcal{T}_0$ and $\mathcal{T}_1$ be trivializations of $\gamma(0)^{*}\mathcal{G}$ and $\gamma(1)^{*}\mathcal{G}$, respectively. Note that, by definition of the connection $\omega_{\mathcal{G}}$ on $\trcon_{\mathcal{G}}$, 
\begin{equation}
\label{eq:surfhol}
\tau^{\omega_{\mathcal{G}}}_{\gamma}(\mathcal{T}_0) = \mathcal{T}_1 \cdot \mathscr{A}_{\mathcal{G}}(\exd\gamma,\mathcal{T}_0,\mathcal{T}_1)\text{,}
\end{equation}
where $\mathscr{A}_{\mathcal{G}}(\exd\gamma, \mathcal{T}_0,\mathcal{T}_1)$ is the surface holonomy of $\mathcal{G}$. The crucial part of the proof is the calculation of this surface holonomy. 
Let $e: [0,1] ^2 \to [0,1] \times S^1$ be defined by $e(t,s) :=(t,s-\frac{1}{2})$, and let $\Phi := \exd\gamma \circ e$. In \cite[Section 6.2]{waldorf10} we have constructed a  lift $\tilde\Phi:[0,1] ^2  \to \px Mx$ along the evaluation map. Let $\sigma_u\in PLM$ be defined by $\sigma_u(t) \eq l(\tilde\Phi(t,0),\tilde\Phi(t,1))$, and let $h_0,h_1 \in PLM$ be thin paths from $h_k(0)=\gamma(k)$ to $h_k(1) = \sigma_u(k)$.  
We have \cite[Lemma 6.2.3]{waldorf10}:
\begin{equation}
\label{eq:holglue}
\mathscr{A}_{\mathcal{G}}(\exd\gamma,\mathcal{T}_0,\mathcal{T}_1)^{-1} =   \exp \left (  \int_{[0,1]^2} \tilde\Phi^{*}B \right ) \cdot  \mathrm{PT}(\sigma_u^{*}P,\tau_{h_0}(\varphi(\mathcal{T}_0)),\tau_{h_1}(\varphi(\mathcal{T}_1))) \text{,}
\end{equation}
where the last term is an element in~$A$ defined by
\begin{equation}
\label{eq:pt}
\ptr{\sigma_u}^{\omega}(\tau_{h_0}(\varphi(\mathcal{T}_0))) \cdot  \mathrm{PT}(\sigma_u^{*}P,\tau_{h_0}(\varphi(\mathcal{T}_0)),\tau_{h_1}(\varphi(\mathcal{T}_1))) =\tau_{h_1}(\varphi(\mathcal{T}_1))\text{,}
\end{equation} 
and $\tau_{\sigma_u}^{\omega}$ denotes the parallel transport of $P$ along  $\sigma_u \in PLM$. We have to compute the two terms on the right hand side of \erf{eq:holglue}.

First we show that the integral in \erf{eq:holglue} vanishes.  Indeed, since $\tilde \Phi$ is a lift of the rank one map $\Phi$ through the \emph{subduction} $\ev_1:\px Mx \to M$, it follows that $\tilde\Phi$ is also a rank one map. Hence, $\tilde\Phi^{*}B=0$ \cite[Lemma A.3.2]{waldorf9}.

For the second term, we claim that the diagram
\begin{equation}
\label{eq:assA}
\alxydim{@=1.5cm}{P_{\gamma(0)} \ar[d]_{\ptr{h_0}^{\omega}} \ar[r]^{\ptr{\gamma}^{\omega}} & P_{\gamma(1)} \ar[d]^{\ptr{h_1}^{\omega}} \\ P_{\tilde\Phi(0,0),\tilde\Phi(0,1)} \ar[r]_{\tau^{\omega}_{\sigma_u}}  & P_{\tilde\Phi(1,0),\tilde\Phi(1,1)} }
\end{equation}
of parallel transport maps in $P$ is commutative. Indeed, all four paths are thin, and the connection $\omega$  is thin. The paths $\gamma$, $h_1$ and $h_2$ are thin by assumption, and the path $\sigma_u$ is thin because the map
\begin{equation*}
[0,1]^2 \times [0,1] \to M : (t,s,r) \mapsto \tilde\Phi(t,s)(r)
\end{equation*} 
is of rank one, which can be checked explicitly  using the definition of $\tilde\Phi$ and using  that $\Phi$ is a rank one map.

\putindent{Summarizing, we obtain}
\begin{multline*}
\ptr{h_1}^{\omega}(\ptr{\gamma}^{\omega}(\varphi(\mathcal{T}_0))) \stackrel{\erf{eq:assA}}{=} \ptr{\sigma_u}^{\omega}(\tau^{\omega}_{h_0}(\varphi(\mathcal{T}_0))) \stackrel{\erf{eq:pt}}{=} \tau^{\omega}_{h_1}(\varphi(\mathcal{T}_1)) \cdot \mathrm{PT}(\sigma_u^{*}P,\tau_{h_0}(\varphi(\mathcal{T}_0)),\tau_{h_1}(\varphi(\mathcal{T}_1)))^{-1} \\\stackrel{\erf{eq:holglue}}{=} \tau^{\omega}_{h_1}(\varphi(\mathcal{T}_1)) \cdot \mathscr{A}_{\mathcal{G}}(\exd\gamma,\mathcal{T}_0,\mathcal{T}_1)    \stackrel{\erf{eq:surfhol}}{=}  \tau^{\omega}_{h_1}(\varphi(\ptr{\gamma}^{\omega_{\mathcal{G}}}(\mathcal{T}_0)))\text{.}
\end{multline*}
This shows that $\varphi$ exchanges the parallel transport of the connections $\omega$ and $\omega_{\mathcal{G}}$ along thin paths. 
\end{proof}

\begin{lemma}
\label{lem:essinj}
The functor $h\un_x$ is essentially injective, i.e.
if $P_1$ and $P_2$ are thin fusion bundles over $LX$, and  $h\un_x(P_1)$ and $h\un_x(P_2)$ are isomorphic, then  $P_1$ and $P_2$ are isomorphic. \end{lemma}

\begin{proof}
We choose good connections on $\un_x(P_1)$ and $\un_x(P_2)$, and denote the corresponding bundle gerbes with connections by $\mathcal{G}_1$ and $\mathcal{G}_2$, respectively.  Lemma \ref{lem:connmorph} and the assumption ensure the existence of a 2-form $\eta \in \Omega^2_{\mathfrak{a}}(M)$ and a connection-preserving 1-isomorphism
\begin{equation*}
\mathcal{A}: \mathcal{G}_1 \to \mathcal{G}_2 \otimes \mathcal{I}_{\eta}\text{.}
\end{equation*} 
Since transgression is a monoidal functor, $\mathcal{A}$ transgresses to an isomorphism
\begin{equation*}
\trcon_{\mathcal{A}}: \trcon_{\mathcal{G}_1} \to \trcon_{\mathcal{G}_2} \otimes \trivlin_{L\eta}
\end{equation*} 
of fusion bundles with connection.
Applying the functor $\sufi$ and using the isomorphism of Lemma \ref{lem:transreg} as well as Proposition \ref{prop:sftrivial}, we get the claimed isomorphism $P_1 \cong P_2$. 
\end{proof}

Now we can conclude with
the following.
\begin{proposition}
\label{prop:fullf}
The functor $h\un_x$ is full and faithful. 
\end{proposition}

\begin{proof}
Fix two thin fusion bundles $P_1$ and $P_2$  over $LM$.  Lemma \ref{lem:essinj} shows that the Hom-set $\hom(P_1,P_2)$ is empty if and only if the Hom-set $\hom(h\un_x(P_1),h\un_x(P_2))$ is empty. In the case that both are non-empty, the first is a torsor over the group $\fushom A{\mathcal{L}X}$ by Proposition \ref{prop:homtorsor}, and the second is a torsor over the group $\hc 0 \diffbun AX$ \cite[Lemma 3.1.2 (iii)]{waldorf10}. Both groups are isomorphic by \cite[Theorem A]{waldorf9}, and $\un_x$ is equivariant by Proposition \ref{prop:equiv}. Thus, $\un_x$ is an equivariant map between torsors, and so a bijection.
\end{proof}

\begin{appendix}

\setsecnumdepth{1}

\section{An Exact Sequence for Subductions over Manifolds}

\label{sec:app}

The following lemma was originally proved by Murray \cite{murray} for a surjective submersion $\pi:Y \to M$ between smooth manifolds. It can be generalized to a diffeological space $Y$ with a subduction $\pi:Y \to M$ to a smooth manifold $M$; see  \cite[Definition A.2.1]{waldorf9} for the definition of a subduction. 

We denote for $p \geq 1$ and $1 \leq i \leq p$ by $\partial_i: Y^{[p]} \to Y^{[p-1]}$ the map that omits the $i$-th factor. For $q\geq 0$, we get a linear map
\begin{equation}
\label{eq:delta}
\delta: \Omega^q(Y^{[p]}) \to \Omega^q(Y^{[p+1]}): \omega \mapsto \sum_{i=0}^{p+1} (-1)^{p} \partial_i^{*}\omega\text{.}
\end{equation}

\begin{lemma}
\label{lem:exactsequence}
For $\pi:Y \to M$ a subduction between a diffeological space $Y$  and a smooth manifold $M$, and $q \in \N$ , the sequence
\begin{equation*}
\alxydim{}{0 \ar[r] & \Omega^q(M) \ar[r]^{\pi^{*}} & \Omega^q(Y) \ar[r]^{\delta} & \Omega^q(Y^{[2]}) \ar[r]^{\delta} & \Omega^q(Y^{[3]}) \ar[r] & \hdots}
\end{equation*}
is exact. 
\end{lemma}

\begin{proof}
We rewrite Murray's original proof in such a way that it passes to the diffeological setting. It is clear that $\delta \circ \delta = 0$. Let $\omega \in \Omega^q(Y^{[p+1]})$ such that $\delta\omega = 0$. We have to construct $\rho \in \Omega^q(Y^{[p]})$ such that $\delta\rho = \omega$. First, we assume that $\pi:Y \to M$ has a global smooth section $s: M \to Y$. Such a section defines a smooth map 
\begin{equation*}
s_p: Y^{[p]} \to Y^{[p+1]}: (y_1,..,y_p) \mapsto (y_1,..,y_p,s(x))\text{,}
\end{equation*}
where $x := \pi(y_i)$ for all $i$. Then, we define $\rho := (-1)^p s_p^{*}\omega$. The identity $\delta\rho = \omega$  follows from the assumption $\delta\omega = 0$ and the relations between the maps $s_p$ and the projections.

In case $\pi:Y \to M$ has no global section, let $U_{\alpha}$ be a cover of $M$ by open sets that admit local smooth sections $s_\alpha: U_{\alpha} \to Y$. Such covers exist because $\pi$ is a subduction \cite[Lemma A.2.2]{waldorf9}. Setting  $Y_{\alpha} := \pi^{-1}(U_{\alpha})$   one is over each open set $U_{\alpha}$ in the case discussed above, and a standard partition of unity argument completes the proof. 
\end{proof}

\setsecnumdepth{1}

\section{Additional Remark about Thin Structures}

\label{sec:6}

In this section I would like to mention a surprising consequence of Theorem \ref{th:mainA} concerning the notion of a thin structure, namely the fact that a thin structure is -- up to isomorphism -- no additional structure at all.

\begin{proposition}
\label{prop:superficialprop}
If $(P,\lambda)$ is a fusion bundle, and $d_1$, $d_2$ are compatible, symmetrizing thin structures, then there exists a fusion-preserving, thin bundle isomorphism 
\begin{equation*}
\varphi \maps (P,\lambda,d_1) \to (P,\lambda,d_2)\text{.}
\end{equation*}
\end{proposition}

\begin{proof}
The regression functor $h\un_x: \fusbunth A{LM} \to \hc 1 \diffgrb AM$ forgets the thin structure, but is an equivalence of categories. 
\end{proof}

Below we construct the isomorphism $\varphi$ explicitly. 
Before that, let me remark that in spite of the statement of Proposition \ref{prop:superficialprop} thin structures are indispensable. First, the \emph{existence} of a thin structure is  a condition. In order to see that, recall that every almost-thin bundle over $LM$ is $\diff^{+}(S^1)$-equivariant, in particular $S^1$-equivariant. In other words, any fusion bundle over $LM$ that is not equivariant does not admit an almost-thin structure. Second, the definition of the Hom-sets in $\fusbunth A{LM}$ uses \emph{fixed  thin structures} on the bundles; the mere existence of thin structures would lead to different Hom-sets and would spoil the equivalence of   Theorem \ref{th:mainA}.

Let us construct the isomorphism $\varphi$ of Proposition \ref{prop:superficialprop} explicitly. We have to employ that the functor $h\un_x$ is essentially injective. Following the proof of Lemma \ref{lem:essinj} we choose compatible, symmetrizing thin connections $\omega_1$, $\omega_2$ integrating the thin structures $d_1$, $d_2$, respectively. Further, we choose good connections  on $\un_x(P,\lambda)$, one with respect to $\omega_1$ and another one with respect to $\omega_2$, and denote the corresponding bundle gerbes with connections by $\mathcal{G}_1$ and $\mathcal{G}_2$, respectively.  From Lemma \ref{lem:transreg} we get  fusion-preserving thin bundle isomorphisms
\begin{equation*}
\varphi_1: \sufi(\trcon_{\mathcal{G}_1}) \to (P,\lambda,d_1)
\quand
\varphi_2:  \sufi(\trcon_{\mathcal{G}_2}) \to (P,\lambda,d_2) \text{.}
\end{equation*}
Since $\mathcal{G}_1 = \mathcal{G}_2$ as bundle gerbes \emph{without} connections, there exist a 2-form $\rho \in \Omega^2_{\mathfrak{a}}(M)$ and a connection $\eta$ on the identity morphism $\id: \mathcal{G}_1 \to \mathcal{G}_2$ such that
\begin{equation}
\label{eq:idiso}
\id_{\eta}:  \mathcal{G}_1 \to \mathcal{G}_2 \otimes \mathcal{I}_{\rho}
\end{equation}
is a connection-preserving 1-isomorphism; see Lemma \ref{lem:connmorph}. In more detail, let $\varepsilon \in \Omega_{\mathfrak{a}}^1(LM)$ be the difference between the connections $\omega_1$ and $\omega_2$, i.e. $\omega_2 = \omega_1 + \varepsilon$.   That $\omega_1$ and $\omega_2$ are compatible with the same fusion product implies $\delta(l^{*}\varepsilon) = 0$, where $\delta$ is the operator \erf{eq:delta} for the subduction $\ev_1: \px Mx \to M$. By Lemma \ref{lem:exactsequence} there exists $\eta \in \Omega^1_{\mathfrak{a}}(\px Mx)$ such that $l^{*}\varepsilon = \delta\eta$, and any choice of $\eta$  is a connection on $\id$. 
All together, we obtain a fusion-preserving, thin bundle isomorphism
\begin{equation*}
\alxydim{@C=1.7cm}{(P,\lambda,\omega_1) \ar[r]^-{\varphi_1^{-1}} & \trcon_{\mathcal{G}_1} \ar[r]^-{\trcon_{\id_{\eta}}} & \trcon_{\mathcal{G}_2} \otimes \trivlin_{-L\rho} \ar[r]^-{\varphi_2 \otimes \id} &  (P,\lambda,\omega_2) \otimes \trivlin_{-L\rho}\text{.}}
\end{equation*}
Under the functor $\sufi$ it yields with Proposition \ref{prop:sufitriv} the claimed isomorphism $\varphi$.

Now we will be more explicit. The bundle morphism $\varphi$ differs from the identity $\id\maps P \to P$ by the action of a smooth map $f:LM \to A$, i.e. $\id \cdot f = \varphi$. Since both $\id$ and $\varphi$ preserve the  fusion product $\lambda$, the map $f$ satisfies the fusion condition. (It is not a fusion map because it is not constant on thin homotopy classes of loops.)
We shall compute the map $f$. 

For a loop $\beta\in LM$ let $\mathcal{T}$ be a trivialization of $\beta^{*}\mathcal{G}_1$, consisting of a principal $A$-bundle $T$ with connection over $Z:= S^1 \lli{\beta}\times_{\ev_1} \px Mx$, and of a connection-preserving isomorphism $\tau: l^{*}P \otimes \zeta_2^{*}T \to \zeta_1^{*}T$. Let $\mathcal{T}' =(T',\tau')$ be the trivialization of $\beta^{*}\mathcal{G}_2$ given by $T' := T \otimes \trivlin_{\mathrm{pr}^{*}\eta}$ and $\tau' = \tau$, where $\mathrm{pr}:Z \to \px Mx$ and $\tau'$ is still connection-preserving because of $l^{*}\varepsilon = \delta\eta$.  We claim that
\begin{equation*}
\trcon_{\id_{\eta}}(\mathcal{T}) = \mathcal{T}' \otimes (\beta,1) \in \trcon_{\mathcal{G}_2} \otimes \trivlin_{-L\rho}\text{.}
\end{equation*}
This means
\begin{equation*}
\varphi_1(\mathcal{T}) \cdot f(\beta) = \varphi_2(\mathcal{T}')\text{.}
\end{equation*}

Now we recall the definition of $\varphi_1(\mathcal{T})$ and $\varphi_2(\mathcal{T}')$ from \cite[Section 6.2]{waldorf10}. We have to choose a path $\gamma \in \px Mx$ with $\gamma(1)=\beta(0)$, together with two paths $\gamma_1$, $\gamma_2$; further an element $q \in T_1|_{\frac{1}{2},\gamma_1 \pcomp \gamma}$. Let $\alpha_i \in PZ$ be the canonical path with $\alpha_i(0)=(\id \pcomp \gamma,0)$ and $\alpha_i(1)=(\gamma_i \pcomp \gamma,\frac{1}{2})$, and let $\alpha := \alpha_2 \pcomp \prev{\alpha_1}$. Then, elements $p_1,p_2 \in P_{l(\gamma_1 \pcomp \gamma,\gamma_2 \pcomp \gamma)}$ are determined by
\begin{equation*}
\tau(p_1 \otimes \tau_{\alpha}(q))=q
\quand
\tau'(p_2 \otimes \tau'_{\alpha}(q))=q\text{,}
\end{equation*} 
where $\tau_{\alpha}$ and $\tau_{\alpha}'$ denote the parallel transport in $T$ and $T'$, respectively. We obtain 
\begin{equation*}
p_2 = p_1 \cdot \exp \left ( \int_{\tilde\alpha} \eta \right )\text{,}
\end{equation*}
where $\alpha' := \mathrm{pr}(\alpha) \in P\px Mx$. 
Let $h \in PLX$ be a thin homotopy between $l(\gamma_1 \pcomp \gamma, \gamma_2 \pcomp \gamma)$ and $\beta$. Then, we have by definition
\begin{equation*}
\varphi_1(\mathcal{T}) = \tau^{\omega_1}_{h}(p_1)
\quand
\varphi_2(\mathcal{T}') = \tau^{\omega_2}_h(p_2)\text{.}
\end{equation*}
Summarizing, we get
\begin{equation*}
f(\beta) := \exp \left (  \int_{\tilde \alpha}\eta + \int_{h} \varepsilon \right )\text{.}
\end{equation*}
One may now  double-check independently  that $f$ is well-defined and smooth, and that the bundle morphism $\varphi = \id \cdot f$ is fusion-preserving and thin. These checks only employ the fact that $\varepsilon$ is (by definition) a compatible, symmetrizing and thin connection on the trivial fusion bundle. 

\begin{remark}
At first sight, it may seem that $f$ is homotopic to the constant map $1$, by scaling the forms $\varepsilon$ and $\eta$ to zero. However, the proof that the above definition of $f$ is independent of the choices of $\gamma$ and $h$ uses that $\epsilon$ is a \emph{thin} and \emph{symmetrizing} connection on the trivial fusion bundle, and  both properties are not scaling invariant. Thus, $f$ is not homotopic to the constant map.  
\end{remark}

\end{appendix}

\tocsection{Table of Notations}

For the convenience of the reader we include a table of notations from all parts of this paper series. The roman letters refer to \cite{waldorf9}, \cite{waldorf10}, and the present paper (III). 

\newcommand{\notation}[3]{
  \noindent
  \begin{minipage}[t]{0.20\textwidth}#1\end{minipage}
  \begin{minipage}[t]{0.59\textwidth}#2\vspace{0.3cm}\end{minipage}\hfill
  \begin{minipage}[t]{0.2\textwidth}\begin{flushright}#3\end{flushright}\end{minipage}}

\notation{$PX$  ($\pt X$)}{The diffeological space of (thin homotopy classes of) paths with sitting instants}{I, Section 2.1}

\notation{$\px Xx$ ($\ptx Xx$)}{The diffeological space of (thin homotopy classes of) paths with sitting instants starting at a point $x\in X$}{I, Section 4.1}

\notation{$BX$ ($\mathcal{B}X$)}{The diffeological space of (thin homotopy classes of) bigons in $X$}{II, Appendix A}

\notation{$LX$}{The free loop space of $X$, considered as a diffeological space}{I, Section 2.2}

\notation{$\mathcal{L}X$}{The thin loop space of $X$}{I, Section 2.2}

\notation{$\mathfrak{L}X$}{The thin loop stack of $X$}{III, Section \ref{sec:thin}}

\notation{$\fus A{\mathcal{L}X}$}{The group of fusion maps on $\mathcal{L}X$}{I, Section 2.2}

\notation{$h\fus A{\mathcal{L}X}$}{The group of homotopy classes of fusion maps on $\mathcal{L}X$}{I, Section 2.2}

\notation{$\fuslc A{\mathcal{L}X}$}{The group of locally constant fusion maps on $\mathcal{L}X$}{I, Section 4.2}

\notation{$\bun AX$}{The groupoid of (diffeological) principal $A$-bundles over $X$}{I, Section 3.1}

\notation{$\buncon AX$}{The category of (diffeological) principal $A$-bundles with connection over $X$}{I, Section 3.2}

\notation{$\bunth A{LX}$}{The category of thin bundles over $LX$}{III, Section \ref{sec:thin}}

\notation{$\fusbun A {LX}$}{The category of fusion bundles over $LX$}{II, Section 2.1}

\notation{$\fusbunth A {LX}$}{The category of thin fusion bundles over $LX$}{III, Section 3.2}

\notation{$h\fusbunth A {LX}$}{The homotopy category of thin fusion bundles over $LX$}{III, Section 3.3}

\notation{$\fusbunconsf A {LX}$}{The category of fusion bundles with superficial connections over $LX$}{II, Section 2.2}

\notation{$\diffgrb AX$}{The 2-category of diffeological $A$-bundles gerbes over $X$}{II, Section 3.1}

\notation{$\diffgrbcon AX$}{The 2-category of diffeological $A$-bundles gerbes with connection over $X$}{II, Section 3.1}

\notation{$\hc 0\mathcal{C}$}{The set of isomorphism classes of objects of a (2-) category $\mathcal{C}$.}{}

\notation{$\hc 1\mathcal{C}$}{The category of obtained from a 2-category $\mathcal{C}$ by identifying 2-isomorphic 1-morphisms. }{}

\kobib{../../bibliothek/tex}


\begin{thebibliography}{BCM{\etalchar{+}}02}
\addcontentsline{toc}{section}{\refname}

\bibitem[BCM{\etalchar{+}}02]{bouwknegt1}
P.~Bouwknegt, A.~L. Carey, V.~Mathai, M.~K. Murray, and D.~Stevenson,
  \quot{Twisted K-Theory and K-Theory of Bundle Gerbes}.
\newblock {\em Commun. Math. Phys.}, 228(1):17--49, 2002.
\newblock \kobiburl{http://arxiv.org/abs/hep-th/0106194}
\bibitem[BH11]{baez6}
J.~C. Baez and A.~E. Hoffnung, \quot{Convenient categories of smooth spaces}.
\newblock {\em Trans. Amer. Math. Soc.}, 363(11):5789--5825, 2011.
\newblock \kobiburl{http://arxiv.org/abs/0807.1704}
\bibitem[BM94]{brylinski4}
J.-L. Brylinski and D.~A. McLaughlin, \quot{The geometry of degree four
  characteristic classes and of line bundles on loop spaces {I}}.
\newblock {\em Duke Math. J.}, 75(3):603--638, 1994.
\bibitem[Bry93]{brylinski1}
J.-L. Brylinski, {\em Loop spaces, characteristic classes and geometric
  quantization}, volume 107 of {\em Progr. Math.}
\newblock Birkh\"auser, 1993.
\bibitem[Bun02]{bunke2002}
U.~Bunke, \quot{Transgression of the index gerbe}.
\newblock {\em Manuscripta Math.}, 109(3):263--287, 2002.
\newblock \kobiburl{http://arxiv.org/abs/math/0109052}
\bibitem[CJM{\etalchar{+}}05]{carey4}
A.~L. Carey, S.~Johnson, M.~K. Murray, D.~Stevenson, and B.-L. Wang,
  \quot{Bundle gerbes for {C}hern-{S}imons and {W}ess-{Z}umino-{W}itten
  theories}.
\newblock {\em Commun. Math. Phys.}, 259(3):577--613, 2005.
\newblock \kobiburl{http://arxiv.org/abs/math/0410013}
\bibitem[Gom03]{gomi3}
K.~Gomi, \quot{Connections and Curvings on Lifting Bundle Gerbes}.
\newblock {\em J. Lond. Math. Soc.}, 67(2):510--526, 2003.
\newblock \kobiburl{http://arxiv.org/abs/math/0107175}
\bibitem[Ler]{lerman1}
E.~Lerman, \quot{Orbifolds as stacks?}
\newblock Preprint.
\newblock \kobiburl{http://arxiv.org/abs/0806.4160}
\bibitem[Lot02]{Lott2002}
J.~Lott, \quot{Higher-degree analogs of the determinant line bundle}.
\newblock {\em Commun. Math. Phys.}, 230(1):41--69, 2002.
\newblock \kobiburl{http://arxiv.org/abs/math/0106177}
\bibitem[Met]{metzler}
D.~S. Metzler, \quot{Topological and Smooth Stacks}.
\newblock Preprint.
\newblock \kobiburl{http://arxiv.org/abs/math.DG/0306176}
\bibitem[Mur96]{murray}
M.~K. Murray, \quot{Bundle gerbes}.
\newblock {\em J. Lond. Math. Soc.}, 54:403--416, 1996.
\newblock \kobiburl{http://arxiv.org/abs/dg-ga/9407015}
\bibitem[NWa]{Nikolaus}
T.~Nikolaus and K.~Waldorf, \quot{Four equivalent versions of non-abelian
  gerbes}.
\newblock Preprint.
\newblock \kobiburl{http://arxiv.org/abs/1103.4815}
\bibitem[NWb]{Nikolausa}
T.~Nikolaus and K.~Waldorf, \quot{Lifting problems and transgression for
  non-abelian gerbes}.
\newblock Preprint.
\newblock \kobiburl{http://arxiv.org/abs/1112.4702}
\bibitem[ST]{stolz3}
S.~Stolz and P.~Teichner, \quot{The Spinor Bundle on Loop Spaces}.
\newblock Preprint.
\newblock
  \kobiburl{http://people.mpim-bonn.mpg.de/teichner/Math/Surveys_files/MPI.pdf}
\bibitem[Wala]{waldorf9}
K.~Waldorf, \quot{Transgression to loop spaces and its inverse, {I}:
  Diffeological bundles and fusion maps}.
\newblock {\em Cah. Topol. G\'eom. Diff\'er. Cat\'eg.}, to appear.
\newblock \kobiburl{http://arxiv.org/abs/0911.3212}
\bibitem[Walb]{waldorf10}
K.~Waldorf, \quot{Transgression to loop spaces and its inverse, {II}: Gerbes
  and fusion bundles with connection}.
\newblock Preprint.
\newblock \kobiburl{http://arxiv.org/abs/1004.0031}
\bibitem[Wal11]{waldorf13}
K.~Waldorf, \quot{A loop space formulation for geometric lifting problems}.
\newblock {\em J. Aust. Math. Soc.}, 90:129--144, 2011.
\newblock \kobiburl{http://arxiv.org/abs/1007.5373}
\bibitem[Wu12]{Wu2012}
E.~Wu, {\em A Homotopy Theory for Diffeological Spaces}.
\newblock PhD thesis, University of Western Ontario, 2012.
\end{thebibliography}
\end{document}